\documentclass{lmcs}

\usepackage[T1]{fontenc}

\usepackage{hyperref}
\usepackage[all]{xy}
\usepackage{amsmath}
\usepackage{amssymb}
\usepackage{graphicx}
\usepackage{xspace}
\usepackage{subcaption}
\usepackage{bm}
\usepackage{faktor}
\usepackage[dvipsnames]{xcolor}

\newcommand{\cplxclass}[1]{{\sf #1}\xspace}
\newcommand{\PTime}{\cplxclass{P}}
\newcommand{\NP}{\cplxclass{NP}}
\newcommand{\PSPACE}{\cplxclass{PSPACE}}
\newcommand{\NPSPACE}{\cplxclass{NPSPACE}}
\newcommand{\NEXP}{\cplxclass{NEXP}}

\newcommand{\probname}[2][black]{\color{#1}\textsc{#2}}
\newcommand{\probref}[2][black]{\hyperref[prob:#2]{\probname[#1]{#2}}}

\newcommand{\atoms}{\mathbb{A}}
\newcommand{\id}{\mathrm{id}}
\DeclareMathOperator{\aut}{Aut}
\DeclareMathOperator{\aage}{Age_{\atoms}}
\DeclareMathOperator{\supp}{supp}

\newcommand{\direction}[1]{{\textsf{#1}}\xspace}
\newcommand{\NN}{\direction{N}}
\renewcommand{\SS}{\direction{S}}
\newcommand{\EE}{\direction{E}}
\newcommand{\WW}{\direction{W}}
\newcommand{\NW}{\direction{NW}}
\newcommand{\NE}{\direction{NE}}
\newcommand{\SE}{\direction{SE}}
\newcommand{\SW}{\direction{SW}}

\newcommand{\Z}{\mathbb{Z}}

\newcommand{\colA}{{\color{red} \sf R}}
\newcommand{\colB}{{\color{green} \sf G}}
\newcommand{\colC}{{\color{blue} \sf B}}

\begin{document}

\title[Karp's NP-complete Problems over FO-Definable Structures]{Karp's NP-Complete Problems over First-Order Definable Structures}
\titlecomment{This is a revised and extended version of the conference paper~\cite{fossacs26}}

\author[A.~Healy]{Aidan Healy\lmcsorcid{0009-0003-1015-9437}}
\author[B.~Klin]{Bartek Klin\lmcsorcid{0000-0001-5793-7425}}

\address{University of Oxford, UK} 
\email{aidan.healy@cs.ox.ac.uk, bartek.klin@cs.ox.ac.uk}


\begin{abstract}
We determine the decidability of Karp's NP-complete problems on structures which are first-order definable over the theory of equality, also known as orbit-finite sets with atoms or nominal sets.
\end{abstract}

\maketitle              


\section{Introduction}

We wish to take a range of classical decision problems that are usually considered for finite structures, and study them on the class of structures which are infinite but definable by first-order formulas that use equality only. Precise definitions will follow, but let us begin with two illustrative examples. Here and in the following, fix a countably infinite set $\atoms$, whose elements we call atoms.

\begin{exa}\label{ex:1}
Consider $X={\atoms\choose 2}$, the set of two-element sets of atoms, and the family $\mathcal{S}$ of all three-element subsets of $X$ of the form:
\[
	\{ \{a,b\}, \{a,c\}, \{b,c\} \} \qquad \text{for all distinct}\ a,b,c\in \atoms.
\]
This family admits an {\em exact cover}: there is a sub-family of $\mathcal{S}$ that forms a partition of $X$. However, to find such a cover we must break the pleasant symmetry of $\mathcal{S}$. For example, we may fix an enumeration of $X$, and proceed by induction starting with the empty family: in each step, take the first $\{a,b\}$ that has not been covered yet, choose some $c$ such that neither $\{a,c\}$ nor $\{b,c\}$ have been covered, and add this to the family. The limit of this process is an exact cover.
\end{exa}

\begin{exa}\label{ex:2}
Now let $X$ be the disjoint union of $\atoms\choose 2$ and $\atoms$. The family $\mathcal{S}$ of all subsets of $X$ of the form:
\[
	\{ \{a,b\}, \{a,c\}, a \} \qquad \text{for all distinct}\ a,b,c\in \atoms,
\]
does {\em not} admit an exact cover of $X$. To see why, notice that to cover a pair $\{a,b\}$ we must include a set that contains either $a$ or $b$. Any atom in $\atoms$ can be included only once in this way, so it can be used to cover only two pairs. More generally, $k$ atoms can be used to cover only $2k$ pairs. For $k=6$, this is not enough to cover all the ${6\choose 2}=15$ pairs that are built of the $k$ atoms and need to be covered. 
\end{exa}

Both these arguments are quite simple, but substantial enough to make one wonder whether or not the classical problem \probref{ExactCover} is decidable for structures of this kind. Indeed, one of our main results is that it is not.

By ``structures of this kind'' we mean structures (graphs, hypergraphs, formulas, families...) where all components and relations are defined by first-order formulas that only compare atoms for equality. Such structures are usually infinite, but they are finite up to bijective renaming of atoms. We will find it convenient to describe them within the framework of {\em sets with atoms}, but in the parlance of model theory (see e.g.~\cite{hodges}), they are simply relational structures which are first-order interpretable in $(\atoms,=)$.

The purpose of this paper is to determine the decidability, over such structures, of appropriately extended versions of the classical \NP-complete problems listed by Karp in 1972~\cite{Karp72}.

The idea of transporting computational problems from finite to infinite structures has been explored in several variants, with the main difference being the class of infinite structures considered. The broadest setting is that of {\em recursive structures}~\cite{HH96,HL96}, where nodes are natural numbers and arbitrary decidable relations are permitted. Of course, in this setting all nontrivial questions become undecidable, and the focus is on determining where particular problems lie in the arithmetical hierarchy. 

A smaller class is that of {\em automatic structures}~\cite{KN95,BG00,Gra20}, where nodes are represented as words over a finite alphabet, and relations are recognisable by multi-tape finite-state automata that read those words in parallel. For this class, the model checking problem for first-order logic extended with a limited form of second-order quantification is decidable~\cite{KL10}. As a result, natural extensions of a few of Karp's problems become decidable~\cite{KL10,Koc14}, including \probref{Clique}, \probref{SetPacking} and a variant of~\probref{SetCovering}. Several other problems remain undecidable, including \probref{Hamiltonicity}
and \probref{ExactCover}~\cite{KL10,Koc14}. Some problems that are polynomial-time decidable in the finite case, including \probname{$2$-Sat}, \probname{$2$-colorability}~\cite{Koc14}, and checking whether a graph is connected~\cite{BG04}, are undecidable on automatic structures.

Another restricted class is that of {\em doubly periodic structures}~\cite{Bur84}. These are constructed by placing infinitely many copies of a fixed finite structure on a two-dimensional grid, and imposing additional relations on nodes from neighbouring copies only, in a periodic manner. On this class, \probname{$k$-Sat} and \probname{$k$-Colorability} are decidable for $k=2$ and undecidable for $k>2$~\cite{Fre98}, and one could hope that this could be a setting where the \PTime vs.~\NP gap is blown up to the gap between decidable and undecidable problems. However, all doubly periodic structures are automatic, so decidability results mentioned above hold here as well.

As we said above, our focus is on structures which are first-order interpretable in a pure set. In the following we will simply call such structures {\em definable}. 
All such structures are automatic, since the class of automatic structures is closed under first-order interpretations~\cite{BG00}. They are incomparable with doubly periodic ones. (For example, an infinite clique is definable but not doubly periodic, and an infinite square grid is doubly periodic but not definable.) 

A few classical decision problems have been studied in the framework of definable structures. In~\cite{KKOT15}, it was proved that every Constraint Satisfaction Problem for a fixed finite template is decidable in this setting. This includes \probname{$k$-Sat} and \probname{$k$-colorability}, for every $k$. On the other hand, checking whether there exists a homomorphism from one given definable structure to another is undecidable~\cite{KLOT16}. For solvability of systems of linear equations,~\cite{KKOT15} shows decidability over finite fields and where every equation contains finitely many variables. Both these assumptions are dropped in~\cite{GHL22}, at the price of searching for definable solutions only. In~\cite{GHL25}, linear programming over definable structures is shown to be decidable, and integer linear programming undecidable, again under the assumption that only definable solutions are considered.

The decidability landscape of Karp's \NP-complete problems on definable structures turns out to be surprisingly varied. Out of the original 21 problems:
\begin{itemize}
\item Seven are undecidable, including \probref{CNFSat}, \probref{ExactCover}, and \probref{Hamiltonicity}. We show this by a sequence of reductions, starting from a reduction of the Wang tiling problem to~\probref{ExactCover}. 
\item Eleven are decidable. Two of these, \probref{3-Sat} and \probref{Colorability}, are known from~\cite{KKOT15}, and we generalise the technique used there to cover a few more, including \probref{VertexCover}. Decidability of  \probref{Clique} and \probref{SetPacking} is known for automatic structures (see above), but for definable structures we provide more direct arguments.
\item Three problems essentially rely on adding up unboundedly many numbers, and we see no natural way to extend them to an infinite setting. 
\end{itemize}
We also investigate the complexity of the decidable problems. Under a specific input representation chosen here, they turn out to be either \PSPACE-complete or \NEXP-complete.


\section{Preliminaries: definable structures}\label{sec:prelims}

We assume general familiarity with the classic paper~\cite{Karp72} and the basic concepts discussed there such as graphs, cliques, formulas etc. We will recall along the way the 21 computational problems studied there. In this preliminary section we focus on {\em definable structures}, which we will use as instances for those problems. We will work in the framework of {\em sets with atoms}~\cite{atom-book}, also called {\em nominal sets}~\cite{pitts}. 
Our presentation follows~\cite{KKOT15,KLOT16}; see there and~\cite[Chap.~10]{atom-book} for more details.

Let $\atoms$ be a countably infinite set of {\em atoms}. We want to define sets that are somehow built of atoms and are perhaps infinite, but presented in a finite way and highly symmetric under bijective atom renaming. To this end, given some fixed infinite set of atom variables,
an {\em expression} is either a variable or a formal finite (perhaps empty) union of set-builder expressions of the form
\[
	\{ e \mid x_1,\ldots, x_k\in\atoms,\ \varphi \}
\]
where $e$ is an expression (where the $x_i$ as well as other variables may occur), the $x_i$ are bound variables, and $\varphi$ is a first-order formula over the set of atom variables, with equality as the only relation symbol. Free variables in this expression are those free variables in $e$ and $\varphi$ which are not among the $x_1,\ldots, x_k$. We can, optionally, make free variables explicit by writing them in brackets in the usual way. For instance:
\[
	\{ e(\bar{x},\bar{y}) \mid \bar{x}\in\atoms,\ \varphi(\bar{x},\bar{y}) \},
\]
is another way of writing the previous expression. Here, $\bar{x}$, $\bar{y}$ denote tuples of variables.

For an expression $e$ with free variables $V$, any valuation $\sigma:V\to \atoms$ defines a value $X=e[\sigma]$ in an obvious way, by induction on the structure of $e$. This value is either an atom (if $e$ is merely a variable) or a set. We will call $X$ a {\em definable set with atoms}. If the image of $\sigma$ is some (necessarily finite, and perhaps empty if $e$ has no free variables) $S\subseteq\atoms$, then we may say that $X$ is {\em $S$-definable}. $\emptyset$-definable sets are called {\em equivariant}. For example, the equivariant set $\atoms\choose 2$ from Example~\ref{ex:1} is formally defined by the expression
\[
	X = \{\{x\}\cup\{y\} \mid x,y\in\atoms, x\neq y\},
\]
where the subexpressions $\{x\}$ and $\{y\}$ use a little syntactic sugar: the empty list of bound variables, as well as a formula $\varphi$ which is always true, can be elided. One can simplify descriptions of finite sets a bit further to write e.g.~$\{x,y\}$ for $\{x\}\cup\{y\}$, which makes a formal expression for $\mathcal{S}$ in~Example~\ref{ex:1} look quite like the informal one given there:
\[
	\mathcal{S} = \left\{ \bigl\{ \{x,y\}, \{x,z\}, \{y,z\} \bigr\} \mid x,y,z\in\atoms,\ x\neq y\land x\neq z\land y\neq z \right\}.
\]

The above language of expressions is rudimentary, but using standard set-theoretic machinery it is easy to extend it with pairs, tuples, integers, constants from some given finite sets, specific atoms, etc. For example, an ordered pair $(x,y)$ can be formally defined as the Kuratowski pair $\{x,\{x,y\}\}$, and any expression that uses a specific atom $a\in\atoms$ is shorthand for the same expression where a fixed fresh variable $x$ is used everywhere instead of~$a$, together with a valuation $\sigma$ that maps $x$ to $a$. We will use such syntactic sugar without further warning. In particular, in Examples~\ref{ex:1}-\ref{ex:2}, not only the sets $X$ and $\mathcal{S}$ but also the pair $(X,\mathcal{S})$ is definable (in fact, equivariant).

Furthermore (see~\cite[Chap.~10]{atom-book} for a detailed discussion), it is routine to prove by induction on the structure of expressions that definable sets are closed under Boolean combinations, Cartesian products, images and inverse images of definable functions, quotients under definable equivalence relations, and intersections and unions of definable families, and all these constructions are effectively computable as operations on the defining expressions. Also relations such as set equality and membership between definable sets are decidable. Set-builder expressions can also be safely extended with more syntactic sugar by allowing bound variables to range not only over $\atoms$ but over any definable set, and allowing in $\varphi$ set relations $\in$ and $\subseteq$ (in addition to atom equality) and quantifiers of the form $\exists x\in X$ and $\forall x\in X$, where $X$ is a set defined by an expression. For example, for $(X,\mathcal{S})$ as in Examples~\ref{ex:1}-\ref{ex:2}, the sets
\[
	M = \{(x,Y)\mid x\in X, Y\in\mathcal{S}, x\in Y\} \qquad \text{and} \qquad  \mathcal{F} = \{\{Y\in \mathcal{S}\mid x\in Y\} \mid x\in X\}
\]
are definable (and equivariant). Specifically, for Example~\ref{ex:1}, the corresponding formal definitions (still with just a sprinkle of syntactic sugar introduced before) are:
\begin{align*}
M &= \left\{ \left(\{x,y\}, \bigl\{ \{x,y\}, \{x,z\}, \{y,z\} \bigr\}\right) \mid x,y,z\in\atoms,\ x\neq y\land x\neq z\land y\neq z \right\}, \\
\mathcal{F} &= \left\{ \left\{ \bigl\{ \{x,y\}, \{x,z\}, \{y,z\} \bigr\} \mid z\in\atoms, z\neq x \land z\neq y \right\} \mid x,y\in\atoms, x\neq y \right\}.
\end{align*}
In the following we will use the extended syntax of expressions liberally.

Definable structures are highly symmetric. The group $\aut(\atoms)$ of {\em atom automorphisms} (meaning: arbitrary bijections on $\atoms$) has a canonical action on the class of definable sets: for a definable set $X=e[\sigma]$ and $\pi\in\aut(\atoms)$, define $X\cdot\pi = e[\sigma;\pi]$, where $\sigma;\pi$ denotes function composition: $(\sigma;\pi)(v) = \pi(\sigma(v))$. This amounts to consistently renaming all the atoms throughout $X$ according to $\pi$. For a finite $S\subseteq\atoms$, a $\pi\in\aut(\atoms)$ is called an {\em $S$-automorphism} if $\pi(a)=a$ for all $a\in S$. We say that $S$ {\em supports} a definable set $X$ if $X\cdot\pi=X$ for every $S$-automorphism $\pi$. It is easy to see that every $S$-definable set is supported by $S$, so every definable set has a finite support. Finite supports of a given set are closed under intersection (see~\cite[Prop.~2.3]{pitts} or~\cite[Thm.~4.13]{atom-book}), so every definable set $X$ has a {\em least support}, denoted $\supp(X)$.

We say that sets $X$ and $Y$ are $S$-equivalent if there is an $S$-automorphism $\pi$ such that $X\cdot\pi=Y$. This is an equivalence relation, and its equivalence classes are called {\em $S$-orbits}, or simply {\em orbits} if $S=\emptyset$. Every definable set $X$ is {\em orbit-finite}, i.e., it is a finite union of $\supp(X)$-orbits. For example:
\begin{itemize}
\item $\atoms$ is a single-orbit set. For any finite $S\subseteq \atoms$, the set $S$ has $|S|$ $S$-orbits (every element of $S$ is a singleton orbit), and $\atoms\setminus S$ has one $S$-orbit.
\item For any $k$, the sets $\atoms^{(k)}$ (of non-repeating $k$-tuples) and $\atoms\choose k$ (of sets of size~$k$) are single-orbit sets. More generally, for any finite permutation group $G\leq\textrm{Sym}(k)$, the set $\atoms^{(k)}/G$ of non-repeating $k$-tuples of atoms up to permutations from $G$, is a single-orbit set.
\item The set $\atoms^2$ (of possibly repeating pairs) consists of two orbits: $\atoms^{(2)}\subseteq \atoms^2$ and its complement. More generally, $\atoms^k$ has number of orbits equal to the $k$-th Bell number.
\item The set of finite subsets of $\atoms$ is not orbit-finite, as sets of different sizes fall into different orbits; therefore this set is not definable.
\end{itemize}

There is an alternative but equivalent way to introduce definable sets, where actions of $\aut(\atoms)$ and finite supports are the basic concepts, with orbit-finiteness imposed as an additional condition. It then becomes a representation theorem that every $S$-supported orbit-finite set is (in an $S$-suppported bijection with) an $S$-definable set. This approach is taken in~\cite{atom-book,BFKM24,BKL14}. Other representations exist: as shown in~\cite{BKL14}, every equivariant single-orbit set is in equivariant bijection with a set of the form $\atoms^{(k)}/G$ as mentioned above. This implies that definable structures are first-order interpretable (in the sense of model theory) in the pure set $(\atoms,=)$.

So far we have focused on {\em equality atoms}, where $\atoms$ is a pure set, without any structure imposed on the atoms. This is our main subject of study here, and in the following we will study Karp's problems only on structures definable over equality atoms. However, in Section~\ref{sec:amenability}, as in~\cite{KKOT15}, we will need to make a brief excursion to a richer structure of {\em ordered atoms}, where $\atoms=(\mathbb{Q},\leq)$ is the total order of rationals, with $\aut(\atoms)$ restricted to order-preserving bijections. This extends the language of formulas $\varphi$ in set-builder expressions with a binary order relation $\leq$.  The notions of definability, support and orbit-finiteness are defined as for equality atoms, and the representation theorems mentioned above work analogously. The general framework of sets over any relational structure of atoms is studied in detail in~\cite{BKL14,atom-book}.


\section{Undecidable problems}

We will list Karp's problems that become undecidable over definable structures.

\subsection{Exact cover}

\begin{prob}[\probname{ExactCover}]\label{prob:ExactCover}
\begin{description}
\item[Input] Definable set $X$, definable set $\mathcal{S}$ of subsets of $X$
\item[Question] Is there a pairwise-disjoint subset of $\mathcal{S}$ whose union is $X$?
\end{description}
\end{prob}

In the classical setting, where both $X$ and $\mathcal{S}$ are finite, Karp~\cite{Karp72} proved \NP-hardness of this problem by reduction from graph colorability. Looking for a similar reduction here would be pointless, since colorability of definable graphs is decidable (see~\cite[Thm.~3]{KKOT15} or Theorem~\ref{thm:colorability} below). Instead, we prove undecidability by a reduction from the well-known domino tiling problem, posed by Wang~\cite{Wang61} and shown undecidable by Berger~\cite{Berger66}. 

To describe that problem, fix constants \NN, \SS, \EE and \WW. For a finite set $C$ of colors, a~{\em tile} is a function $t:\{\NN,\SS,\EE,\WW\}\to C$, and for a set $T$ of such tiles, a {\em tiling} of the plane is a~function $\tau:\Z^2\to T$ such that, for all $a,b\in\Z$:
\[
	\tau(a,b)(\NN) = \tau(a,b+1)(\SS) \qquad \text{and} \qquad \tau(a,b)(\EE) = \tau(a+1,b)(\WW).
\]

\medskip

\begin{prob}[\probname{Domino}]\label{prob:Domino}
\begin{description}
\item[Input] A finite set $C$ of colors, a finite set $T$ of tiles over $C$
\item[Question] Is there a tiling with $T$?
\end{description}
\end{prob}

We will reduce~\probref{Domino} to~\probref{ExactCover}. Given an input $(C,T)$, we will construct definable $X$ and $\mathcal{S}$ such that an exact cover of $(X,\mathcal{S})$ exists if and only if a tiling of the plane does.

For some intuition, imagine an infinite clique with all atoms as vertices. The set $X$ will contain four elements for each atom (the {\em vertex elements}), and a few elements for each pair of distinct atoms (the {\em edge elements}). The family $\mathcal{S}$ will contain two kinds of (finite) subsets of $X$. An {\em edge set} will contain all the edge elements associated to a specific pair of atoms. A \emph{tile set}, defined over a quadruple of distinct atoms $a,b,c,d$, will contain some of the vertex elements associated to $a,b,c$ and $d$, as well as some edge elements associated to the pairs $(a,b), (b,c), (c,d)$ and $(d,a)$. We can consider it associated with the quadrilateral $a,b,c,d$.

Only tile sets contain vertex elements, so an exact cover in $\mathcal{S}$ must use infinitely many of them. Each tile set will only partially fill each of its four edges (i.e.~it will not include an entire edge set), so in an exact cover each tile set must match other tile sets complementing it on its edges, in a way that can be unrolled to a tiling of the infinite square grid. Edge sets are used to cover all the edge elements that do not correspond to edges in that grid.

We shall now make these ideas precise.

\begin{thm}\label{thm:undecidability-exact-cover}
\probref{ExactCover} is undecidable.
\end{thm}
\begin{proof}
By reduction from \probref{Domino}. Given as input finite sets $C$ of colors and $T$ of tiles, we shall construct an instance $(X,\mathcal{S})$ of \probref{ExactCover}. Define:
\begin{equation*}
\begin{array}{ll}
X =& \hspace{16pt}\displaystyle \bigcup_{a \in \atoms}\ {\{\NW_a, \NE_a, \SE_a, \SW_a\}} \hspace{117pt} \Bigr\}\ \text{\small vertex elements} \\[15pt]
   &\hspace*{-10pt}
   \scalebox{0.75}[1]{$\displaystyle \left.
   \scalebox{1.33}[1]{$\displaystyle \begin{array}{l}
   \displaystyle  \cup  \bigcup_{(a,b) \in \atoms^{(2)}}{\{\rightarrow_{(a,b)}\}} \\[18pt]
    \displaystyle \cup  \bigcup_{\{a,b\} \in \binom{\atoms}{2}}{\{1_{\{a,b\}}, 2_{\{a,b\}}, 3_{\{a,b\}}, 4_{\{a,b\}}\} \cup \{c_{\{a,b\}} : c \in C\}}.
    \end{array}$}
     \right\}$}\ \text{\small edge elements}
\end{array}
\end{equation*}
Symbols $\NW,\NE,\SE,\SW,\rightarrow,1,2,3$ and $4$ above are merely tags without any inherent meaning. In a formal expression without syntactic sugar, they would be represented by some fixed and distinct finite sets that do not involve any atoms, and an element such as $\NW_a$  for an atom $a\in\atoms$ would be represented as a pair $\{\NW,a\}$. As a result, $X$ is (in an equivariant bijection with) a disjoint union of four copies of the set $\atoms$, one copy of $\atoms^{(2)}$ and $|C|+4$ copies of $\atoms\choose 2$. Note that $X$ depends on $C$ but not on $T$. 

To define $\mathcal{S}$, first for any $\{a,b\} \in \binom{\atoms}{2}$ define the \emph{edge set} $e_{\{a,b\}}$ by:
\[
e_{\{a,b\}} = \{1_{\{a,b\}}, 2_{\{a,b\}}, 3_{\{a,b\}}, 4_{\{a,b\}}, \rightarrow_{(a,b)}, \rightarrow_{(b,a)}\} \cup \{c_{\{a,b\}} : c \in C\}.
\]
This set, with $|C|+6$ edge elements, contains all the elements of $X$ associated with $\{a,b\}$, $(a,b)$ and $(b,a)$.

Furthermore, for any tile $t\in T$ and atoms $(a,b,c,d)\in\atoms^{(4)}$, define the \emph{tile set} $t_{(a,b,c,d)}$ by:
\begin{align*}
    t_{(a,b,c,d)} = \{\SE_a,& \SW_b, \NW_c, \NE_d\}  \cup {\{\rightarrow_{(a,b)}, \rightarrow_{(b,c)}, \rightarrow_{(c,d)}, \rightarrow_{(d,a)}\}} \\
    & \cup {\{1_{\{a,b\}}, 2_{\{a,b\}},1_{\{b,c\}}, 3_{\{b,c\}},3_{\{c,d\}}, 4_{\{c,d\}},2_{\{d,a\}}, 4_{\{d,a\}}\}} \\
    &\cup \{c_{\{a,b\}} : c \in C \wedge t(\NN) = c\} \cup \{c_{\{b,c\}} : c \in C \wedge t(\EE) = c\}\\
    &\cup \{c_{\{c,d\}} : c \in C \wedge t(\SS) \neq c\} \cup \{c_{\{d,a\}} : c \in C \wedge t(\WW) \neq c\}.
\end{align*}
This set contains $4$ vertex elements associated respectively with atoms $a$, $b$, $c$ and $d$, and $2|C|+12$ edge elements associated with pairs $\{a,b\}$, $\{b,c\}$, $\{c,d\}$ and $\{d,a\}$ (and with their ordered variants). For intuition, the set is best visualised around a square with the atoms $a$, $b$, $c$ and $d$ in its corners; see Figure \ref{fig:tile-set} for an example. 

Let $\mathcal{S}$ contain all the edge and tile sets as above.

\begin{figure}
\centering
\includegraphics[scale=1]{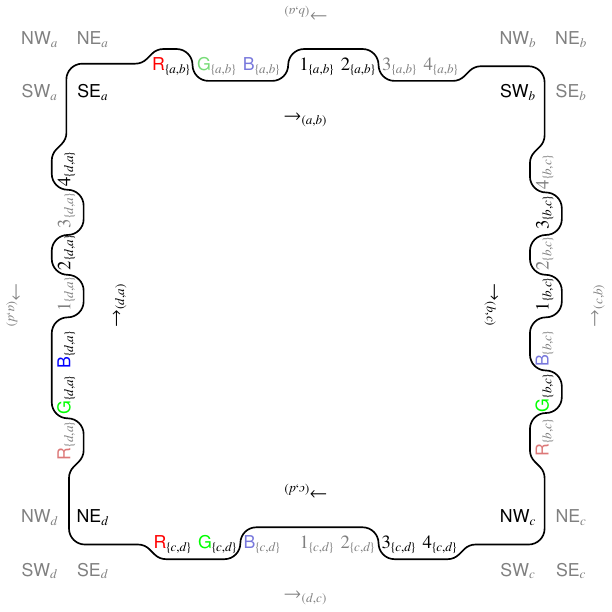}
\caption{The 22-element tile set $t_{(a,b,c,d)}$ for colors $C=\{\colA,\colB,\colC\}$ and a tile $t = (\NN \mapsto \colA, \EE \mapsto \colB, \SS \mapsto \colC, \WW \mapsto \colA)$, with the excluded neighbouring elements marked as gray}
\label{fig:tile-set}
\end{figure}

We shall show that if $(X,\mathcal{S})$ admits an exact cover then $T$ tiles the plane.

Fix an exact cover $E \subseteq \mathcal{S}$. Let $F$ be the set of tile sets in $E$. Since edge sets do not contain any vertex elements, $F$ must be nonempty (in fact, it must be infinite, because there are infinitely many vertex elements to cover, and tile sets are finite). 

So consider any tile set $t_{(a,b,c,d)} \in F$. Note that this set intersects each of the edge sets $e_{\{a,b\}}$, $e_{\{b,c\}}$, $e_{\{c,d\}}$ and $e_{\{d,a\}}$, so neither of these edge sets belongs to $E$.
Since $t_{(a,b,c,d)}$ covers $\rightarrow_{(a,b)}$ but not $\rightarrow_{(b,a)}$, the latter element has to be covered by another set in $E$. Since $e_{\{a,b\}}\not\in E$, that set has to be a tile set; call this unique tile set $n(t_{(a,b,c,d)})$. We can similarly identify tile sets $e(t_{(a,b,c,d)}), s(t_{(a,b,c,d)})$ and $w(t_{(a,b,c,d)})$ in $F$ by looking respectively at the elements $\rightarrow_{(b,c)}$, $\rightarrow_{(c,d)}$ and $\rightarrow_{(d,a)}$ in $t_{(a,b,c,d)}$. Doing this for all tile sets in $F$ defines functions $n,e,s,w: F \rightarrow F$. The names of these functions correspond to four directions on the plane; the intuition is that they will define neighbours in a plane tiling, as shown in Figure~\ref{fig:tile-set-tiling}.

\begin{figure}
\centering
\includegraphics[scale=1.0]{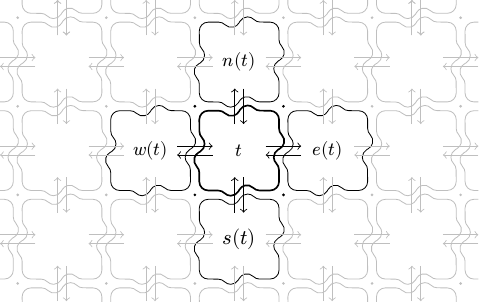}
\caption{Tile sets arranged on a plane}
\label{fig:tile-set-tiling}
\end{figure}

To make this jigsaw-puzzle intuition precise, first note for any $t_{(a,b,c,d)} \in F$ we have $n(t_{(a,b,c,d)}) = t'_{(p,q,b,a)}$ for some tile $t' \in T$ and some atoms $p,q \in \atoms$. Indeed, otherwise $n(t_{(a,b,c,d)})$ would either not contain $\rightarrow_{(b,a)}$, or it would intersect $t_{(a,b,c,d)}$ on at least one of the elements $1_{\{a,b\}}, 2_{\{a,b\}}, 3_{\{a,b\}}, 4_{\{a,b\}}, \rightarrow_{(a,b)}$. Hence $s \circ n = \id$. Similar arguments show $n \circ s = w \circ e = e \circ w = \id$. 

Furthermore, $e(n(t_{(a,b,c,d)}))$ and $n(e(t_{(a,b,c,d)}))$ both contain the vertex element $\NE_b$, so they must be the same tile set. As a result, $n \circ e = e \circ n$. Similarly $n \circ w = w \circ n$, $s \circ e = e \circ s$ and $s \circ w = w \circ s$.

Lastly, note that if $n(t_{(a,b,c,d)})=t'_{(p,q,b,a)}$ then $t(\NN) = t'(\SS)$; otherwise the two tile sets would intersect on $c_{\{a,b\}}$ for $c=t(\NN) \in C$, and neither of them would cover $c_{\{a,b\}}$ for $c=t'(\SS)$. Similarly, if $e(t_{(a,b,c,d)})=t'_{(p,a,d,q)}$ then $t(\EE) = t'(\WW)$.

This is enough to construct a tiling of the plane with tiles from $T$. Indeed, start from any $t_{(a,b,c,d)} \in F$ and define $\phi: \Z^2 \rightarrow F$ by:
\begin{align*}
\phi(0,0) &= t_{(a,b,c,d)} \\
\phi(x,y+1) &= n(\phi(x,y)) \\
\phi(x,y-1) &= s(\phi(x,y)) \\
\phi(x+1,y) &= e(\phi(x,y)) \\
\phi(x-1,y) &= w(\phi(x,y)) 
\end{align*}
for all $x,y \in \Z$. By the equations between the functions $n$, $s$, $e$ and $w$ proved above, this is a well-defined function. Composing $\phi$ with the map that takes each tile set $t_{(a,b,c,d)}$ to its associated tile $t\in T$, we obtain a tiling of $\Z^2$. Note that $\phi$ may not be injective, but this is not a problem. It only implies that the constructed tiling may be periodic.

Next we must show that a tiling gives rise to an exact cover. Fix a tiling $\tau: \Z^2 \rightarrow T$.
Choose a bijection $\alpha: \Z^2 \rightarrow \atoms$ and define:
\begin{align*}
E = \bigl\{&(\tau(x,y))_{(\alpha(x,y+1), \alpha(x+1,y+1), \alpha(x+1,y), \alpha(x,y))} : (x,y) \in \Z^2 \big\}\\
\cup \phantom{|} \big\{&e_{\{a,b\}} : a,b \in \atoms \wedge \lVert \alpha^{-1}(b) - \alpha^{-1}(a) \rVert > 1 \big\},
\end{align*}
where $\lVert \cdot \rVert$ is the Euclidean norm on $\Z^2$.

Intuitively, $\alpha$ arranges the atoms into an infinite grid. The condition on $\lVert \alpha^{-1}(b) - \alpha^{-1}(a)\rVert$ present in the definition of $E$ means simply that atoms $a$ and $b$ are not adjacent in that grid. For every unit square in the grid, $E$ contains the tile set corresponding to the tile $\tau(x,y)$ and to the atoms associated to the four corners of the square. This covers (with pairwise disjoint sets) all vertex elements of $X$, and all edge elements associated to those pairs of atoms which $\alpha$ maps to adjacent points of the grid. The remaining elements of $X$ are all edge elements associated with pairs of atoms not (in this sense) adjacent. These are covered with the edge sets included in $E$.

Finally, $X$ and $\mathcal{S}$ are definable and the construction of $X$ and $\mathcal{S}$ from $C$ and $T$ is effective, so the reduction is complete.
\end{proof}

\begin{rem} \label{rem:exact-cover-with-finite} Note that the family $\mathcal{S}$ in our construction only contains finite sets. Hence \probref{ExactCover} remains undecidable even if we restrict to instances where all sets in $\mathcal{S}$ are finite. On the other hand, we could not impose the dual restriction that every element of $X$ belongs to finitely many sets in $\mathcal{S}$ (in fact, in our construction every element of $X$ belongs to {\em in}finitely many tile sets). Indeed, later we shall see that under that dual restriction the problem ~\probref{ExactCover} becomes decidable.
\end{rem}

For the purposes of further reductions it will sometimes be convenient to assume that the sets in $\mathcal{S}$ have an inbuilt ordering, i.e., that they are (non-repeating) tuples rather than finite sets. Hence consider the following variant of the problem, which will prove convenient in the proofs of Theorems~\ref{thm:undecidability-directed-hamiltonicity} and~\ref{thm:3dmatching}:

\begin{prob}[\probname{ExactTupleCover}]\label{prob:ExactTupleCover}
\begin{description}
\item[Input] Definable set $X$, number $n$, definable set $\mathcal{S}\subseteq X^{(n)}$
\item[Question] Is there $\mathcal{S}' \subseteq \mathcal{S}$ where every $x \in X$ occurs in exactly one tuple?
\end{description}
\end{prob}

\begin{cor}\label{coro:exact-tuple-cover}
    \probref{ExactTupleCover} is undecidable.
\end{cor}
\begin{proof} By reduction from \probref{ExactCover} restricted to instances $(X,\mathcal{S})$ where all sets in $\mathcal{S}$ are finite.

First, note that in such an instance there is an upper bound on the size of a set in $\mathcal{S}$ and that upper bound (call it $n$) is computable from $\mathcal{S}$. We may further assume that all sets in $\mathcal{S}$ are of size exactly $n$: to achieve this, add infinitely many elements to $X$ (e.g.~replace it with a disjoint union of $X$ and $\atoms$), and redefine $\mathcal{S}$ to consist of every subset of size $n$ whose intersection with $X$ is either empty or in (the original) $\mathcal{S}$. This new instance of~\probref{ExactCover} admits an exact cover if and only if the original one does.

Now, let $\mathcal{S'}$ consist of all non-repeating $n$-tuples whose elements exactly list a set in $\mathcal{S}$. This may be produced effectively from $\mathcal{S}$. But then $(X,\mathcal{S'})$ admits an exact tuple cover precisely if $(X,\mathcal{S})$ admits an exact cover, and we are done.
\end{proof}

\subsection{Hitting set, Satisfiability, 0-1 Integer Linear Programming} {~}

\medskip

\noindent
A few problems allow straightforward reductions from \probref{ExactCover}.

\begin{prob}[\probname{HittingSet}]\label{prob:HittingSet}
\begin{description}
\item[Input] Definable set $P$, definable set $\mathcal{F}$ of subsets of $P$
\item[Question] Is there a subset $Q\subseteq P$ such that $|Q\cap C|=1$ for every $C\in \mathcal{F}$?
\end{description}
\end{prob}

In the finite setting~\cite{Karp72}, Karp demonstrated $\NP$-hardness of this problem, noticing that it is essentially \probref{ExactCover} in disguise. The same reduction works here.

\begin{thm}\label{thm:exactcover-hittingset}
\probref{HittingSet} is undecidable.
\end{thm}
\begin{proof}
Given an instance $(X,\mathcal{S})$ of \probref{ExactCover}, define $P=\mathcal{S}$ and let
\[
	\mathcal{F} = \{\{Y\in \mathcal{S}\mid x\in Y\} \mid x\in X\}.
\]
As remarked in Section~\ref{sec:prelims}, $\mathcal{F}$ is definable and its defining expression is computable from that of $X$ and $\mathcal{S}$. This follows from general principles explained e.g.~in~\cite[Chap.~10]{atom-book}, but in this case it is particularly easy to see: the membership relation defines a bipartite graph on points from $X$ and sets from $\mathcal{S}$, and the same bipartite graph, upon swapping the roles of points and sets, describes a family of subsets of $\mathcal{S} = P$. This family is exactly $\mathcal F$. Moreover, exact covers for $(X,\mathcal{S})$ exactly correspond to hitting sets for $(P,\mathcal{F})$.
\end{proof}

\begin{rem}\label{rem:hitting-set-with-finite}
In the above proof, if all sets in $\mathcal{S}$ are finite then we obtain an instance of \probref{HittingSet} where every element of $P$ belongs to finitely many sets in $\mathcal{F}$. Therefore, as a follow-up to Remark~\ref{rem:exact-cover-with-finite}, \probref{HittingSet} is undecidable when restricted to such instances.
\end{rem}

\medskip

A CNF propositional formula over a set $X$ of variables can be represented as a family of subsets of $X\times\{0,1\}$, each of the subsets representing a disjunctive clause. We say that the formula is definable if the family is definable. The notion of a satisfying assignment is as in the finite case. 

\begin{prob}[\probname{CNFSat}]\label{prob:CNFSat}
\begin{description}
\item[Input] Definable CNF formula $\varphi$
\item[Question] Is $\varphi$ satisfiable?
\end{description}
\end{prob}

\begin{thm}\label{thm:undecidability-sat}
\probref{CNFSat} is undecidable.
\end{thm}
\begin{proof}
By reduction from \probref{HittingSet}. Given an instance $(P,\mathcal{F})$, define the following formula over variables from $P$:
\[\textstyle
	\left(\bigwedge_{C\in \mathcal{F}}\bigvee_{p\in C}p\right) \, \land \, \left(\bigwedge_{C\in \mathcal{F}}\bigwedge_{p\neq q\in C}(\neg p \lor \neg q)\right).
\]
Its satisfying assignments correspond to hitting sets for $\mathcal{F}$.
\end{proof}

\begin{rem}\label{rem:CNFfromFO}
There is an alternative proof of undecidability for \probref{CNFSat}, which avoids our reduction from Theorem~\ref{thm:undecidability-exact-cover}, and proceeds directly by a straightforward reduction from the satisfiability problem for the $\forall\exists$-fragment of first order logic~\cite{BGG97}.%
\footnote{\ We are grateful to M.~Boja\'nczyk for pointing this out.} The idea is as follows.

Consider a first-order sentence $\varphi = \forall \overline{x} \exists \overline{y} \; \psi(\overline{x},\overline{y})$, where $\psi$ is quantifier-free and $\overline{x}, \overline{y}$ are disjoint tuples of variables of length $n$ and $k$ respectively. We may assume that $\varphi$ is over a finite relational signature $\sigma$, and that it does not mention equality (so that if it has a model then it has one with a countably infinite domain, hence one with domain $\atoms$).

Let $V$ be the set of basic\footnote{\ In first-order logic these are usually called {\em atomic formulas} or even {\em atoms}; we avoid that terminology for obvious reasons.} formulas $R(a_1, \dots, a_r)$, where $R$ is a relation symbol of arity $r$ in $\sigma$ and the $a_i$ are atoms. These will be variables in our \probref{CNFSat} instance. Note that for every $\overline{a}, \overline{b}$ of appropriate lengths, the formula $\psi(\overline{a},\overline{b})$ can be seen as a propositional formula over $V$. Note also that Boolean valuations $V \rightarrow \{true,false\}$ are in a bijective correspondence with $\sigma$-structures on $\atoms$. 

We intend to write an equivariant orbit-finite CNF formula expressing that the $\sigma$-structure encoded by a valuation is a model of $\varphi$. To do this, it is helpful to add some intermediate variables: let $V'$ consist of one variable $v_{\overline{a}, \overline{b}}$ for each $\overline{a}, \overline{b}$ in (respectively) $\atoms^n, \atoms^k$. Using standard propositional manipulations, transform the following propositional formula over $V\cup \{v_{\overline{a}, \overline{b}}\}$:
\[
	v_{\overline{a}, \overline{b}} \leftrightarrow \psi(\overline{a},\overline{b})
\]
into CNF form. Call the result $\chi_{\overline{a},\overline{b}}$.
Then a $\sigma$-structure on $\atoms$ satisfies $\varphi$ precisely if its corresponding valuation satisfies:
\[
\left(\bigwedge_{\overline{a} \in \atoms^{n}}\bigwedge_{\overline{b} \in \atoms^{k}}\chi_{\overline{a},\overline{b}}\right) \land \left(\bigwedge_{\overline{a} \in \atoms^{n}}{\bigvee_{\overline{b} \in \atoms^{k}}{v_{\overline{a}, \overline{b}}}}\right).
\]
This is a definable CNF formula over $V\cup V'$, effectively derived from $\varphi$.
\end{rem}

Notwithstanding the above construction, our reduction in Theorem~\ref{thm:undecidability-sat} is still worthwhile. In Theorem~\ref{thm:decidability-3-sat} we shall see that \probref{3-Sat}, and more generally \probref{CNFSat} where all clauses are finite, is decidable. The reductions from Theorems~\ref{thm:undecidability-sat} and~\ref{thm:exactcover-hittingset}  therefore imply that, as a counterpoint to Remarks~\ref{rem:exact-cover-with-finite} and~\ref{rem:hitting-set-with-finite}, the following problems are decidable:
\begin{itemize}
\item \probref{HittingSet} for instances $(P,\mathcal{F})$ where all sets in $\mathcal{F}$ are finite,
\item \probref{ExactCover} for instances $(X,\mathcal{S})$ where every $x\in X$ belongs to finitely many sets in $S$.
\end{itemize}

\medskip

One could also try to use the reduction in Remark~\ref{rem:CNFfromFO} to dispense with our reduction from Theorem~\ref{thm:undecidability-exact-cover} altogether. Indeed, there is an easy reduction from \probref{CNFSat} to \probref{HittingSet} (and thus to \probref{ExactCover}, by reversing the construction from Theorem~\ref{thm:exactcover-hittingset}): given a formula $\varphi$, begin by creating a \probref{HittingSet} instance over the variables of $\varphi$ and their formal negations, and add a set $\{v, \neg v\}$ for each $v$. It then suffices to add, for each clause $C$ of $\varphi$, extra elements and sets which force at least one literal in $C$ to be included in any hitting set. This is easy to do: the gadget shown in Figure \ref{fig:cnfsat-hittingset-reduction} (where elements are shown as dots and sets are circled) ensures that at least one of the blue elements is included in every hitting set, and does not introduce any other restrictions on said blue elements.

\begin{figure}
\centering
\includegraphics[scale=0.7]{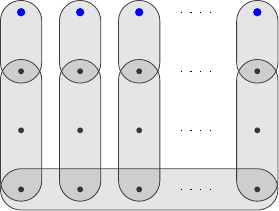}
\caption{A \probref{HittingSet} instance forcing the inclusion of a blue element}
\label{fig:cnfsat-hittingset-reduction}
\end{figure}

This proves that \probref{HittingSet}, and by extension \probref{ExactCover}, are undecidable. However, this construction would not let us show undecidability of \probref{ExactTupleCover} (or \probref{ExactCover} on instances $(X,\mathcal{S})$ where all sets in $\mathcal{S}$ are finite). Indeed, in the CNF formula produced in Remark~\ref{rem:CNFfromFO} some variables are present in infinitely many clauses, hence in the resulting instance $(P,\mathcal{F})$ of~\probref{HittingSet} some elements of $P$ belong to infinitely many sets in $\mathcal{F}$, and when (the reverse of) the construction from Theorem~\ref{thm:exactcover-hittingset} is applied to that, we obtain an instance $(X,\mathcal{S})$ of~\probref{ExactCover} where some sets in $\mathcal{S}$ are infinite. In contrast, our construction from Theorem~\ref{thm:undecidability-exact-cover} does prove Corollary~\ref{coro:exact-tuple-cover}, and it will be an essential starting point of further reductions in Theorems~\ref{thm:undecidability-directed-hamiltonicity} and~\ref{thm:3dmatching} below.

\medskip

We now turn attention to 0-1 Integer Linear Programming. 
For a definable set $X$ of variables, an equation is a pair $(c,n)$, where $c:X\to\Z$ is a definable function and $n\in\Z$. This is intended to represent the equation $\sum_{x\in X}c_xx = n$. A valuation $v:X\to\{0,1\}$ solves the equation if (i) $v(x)$ and $c_x$ are simultaneously non-zero for only finitely many $x$'s, and (ii) $\sum_{x\in X}c_xv(x)$ equals $n$. (The first condition makes the sum well defined.) The problem is then posed as:

\begin{prob}[\probname{0-1-ILP}]\label{prob:0-1-ILP}
\begin{description}
\item[Input] Definable set $X$, definable set $E$ of equations over $X$
\item[Question] Does $E$ have a solution?
\end{description}
\end{prob}

\begin{thm}\label{thm:undecidability-01ILP}
\probref{0-1-ILP} is undecidable.
\end{thm}
\begin{proof}
Given an instance $(P,\mathcal{F})$ of \probref{HittingSet}, put $X=P$ and for each $C\in \mathcal{F}$ add the equation $\sum_{p\in C}p=1$ to $E$. Solutions to $E$ are hitting sets for $\mathcal{F}$.
\end{proof}

This result is not entirely new. In~\cite[Sec.~9]{GHL25}, undecidability is proved for Integer Linear Programming but without the 0-1 restriction, with inequalities allowed in addition to equations, and with the additional condition that only finite solutions are sought. However, an inspection of that proof shows that only 0-1 variables are actually used and inequalities can be encoded as equations. Moreover, the proof can be adapted to deal with arbitrary solutions.\footnote{\ We are grateful to A.~Ghosh, P.~Hofman and S.~Lasota for a helpful discussion about this issue.}

\subsection{Hamiltonicity}

In the finite setting~\cite{Karp72}, Karp considered the existence of a Hamiltonian cycle, both in directed and undirected graphs. It may not be clear what an infinite cycle should mean, but the definition of ``a Hamiltonian cycle'' as ``a connected $2$-regular\footnote{\ In directed graphs we count the total degree, and require in-degree and out-degree to both be $1$.} subgraph that meets every vertex'' generalises well: on an infinite graph, it means a path which is doubly-infinite, i.e.~extends infinitely in both directions. 

\begin{prob}[\probname{(Un)DirectedHamiltonicity}]\label{prob:DirectedHamiltonicity}\label{prob:UndirectedHamiltonicity}\label{prob:Hamiltonicity}
\begin{description}
\item[Input] Definable (un)directed graph $G$
\item[Question] Does $G$ have a doubly-infinite Hamiltonian path?
\end{description}
\end{prob}

\begin{thm}\label{thm:undecidability-directed-hamiltonicity}
\probref{DirectedHamiltonicity} is undecidable.
\end{thm}
\begin{proof}
By a reduction from~\probref{ExactTupleCover}.
Take as input an instance $(X,\mathcal{S})$, with $\mathcal{S}\subseteq X^{(n)}$ a set of non-repeating tuples. We may assume that $\mathcal{S}$ is infinite. Let us construct $G = (V,E)$ which has a doubly-infinite Hamiltonian path if and only if $(X,\mathcal{S})$ admits an exact cover, and may moreover be produced effectively from $(X,\mathcal{S})$.

We intend $G$ to consist of an infinite independent set of vertices corresponding to the elements of $X$, together with a gadget for each tuple $\overline{x} \in \mathcal{S}$. The gadgets will mediate how a doubly-infinite Hamiltonian path can visit vertices corresponding to elements of $X$.

Formally, we define $V$ as the disjoint union:
\begin{align*}
V = X \uplus \left\{(\overline{x},i) \mid \overline{x}\in \mathcal{S},\ 1\leq i\leq 3n+3\right\}.
\end{align*}
The vertices $(\overline{x},i)$ will form the gadget corresponding to $\overline{x}$; the first and the last of them will be called the {\em ends} of the gadget.

For edges, first connect vertices in each gadget as illustrated in Figure~\ref{fig:gadget}; then connect all ends of all gadgets, joining them into a single infinite clique (see Figure \ref{fig:gadget-graph} for an illustration). Formally, we include an edge:
\begin{itemize}
    \item from $(\overline{x},i)$ to $(\overline{x},i+1)$ whenever $i$ is not a multiple of $3$;
    \item from $(\overline{x},i+1)$ to $(\overline{x},i)$ for all $i$;
    \item from $(\overline{x},3i)$ to $x \in X$ whenever $x$ is the $i$th component of $\overline{x}$;
    \item from $x \in X$ to $(\overline{x},3i+1)$ whenever $x$ is the $i$th component of $\overline{x}$;
    \item from $(\overline{x},i)$ to $(\overline{y},j)$ where $i$ is either $1$ or $3n+3$ ($n$ is the arity of $\overline{x}$) and $j$ is either $1$ or $3m+3$ ($m$ the arity of $\overline{y}$).
\end{itemize}

\begin{figure}
\makebox[\textwidth][c]{\includegraphics{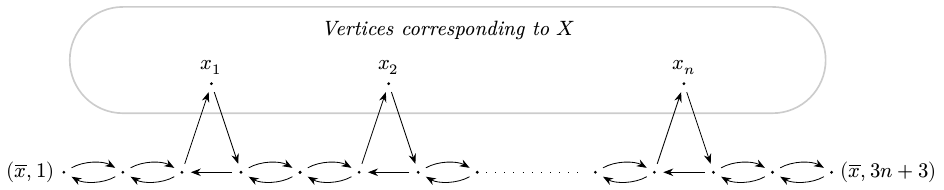}}
\caption{The gadget corresponding to $\overline{x} = (x_1, \dots, x_n)$}
\label{fig:gadget}
\end{figure}

\medskip

\begin{figure}
\centering
\includegraphics{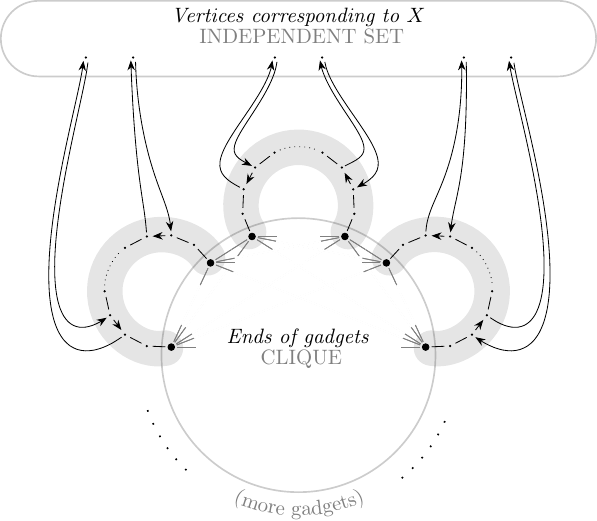}
\caption{The graph constructed in the proof of Theorem \ref{thm:undecidability-directed-hamiltonicity}}
\label{fig:gadget-graph}
\end{figure}

\medskip

\begin{figure}
\begin{subfigure}[h]{0.9\linewidth}
\centering
\includegraphics[scale=1.2]{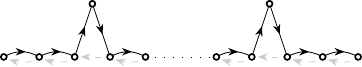}
\caption{From first vertex to last}
\end{subfigure}
\\[1em]
\begin{subfigure}[h]{0.9\linewidth}
\centering
\includegraphics[scale=1.2]{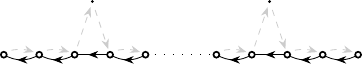}
\caption{From last vertex to first}
\end{subfigure}
\caption{The two ways to visit a gadget}
\label{fig:gadget-visitation}
\end{figure}

Consider a doubly-infinite Hamiltonian path in this graph. Note it can visit the gadget corresponding to $(x_1, \dots, x_n)$ in one of two ways: from first vertex to last, passing through all of $x_1, \dots, x_n$, or from last to first, visiting no $x_i$ along the way (see Figure \ref{fig:gadget-visitation}). In particular, it is not possible to enter or exit the gadget via any of the vertices $x_i$, as that would leave some vertices in the gadget unvisited and inaccessible. As a result, every $x \in X$ has a unique gadget which was being traversed (necessarily, first-to-last) when $x$ was visited. The tuples corresponding to those chosen gadgets form an exact cover.

Conversely, fix an exact cover $E$. The set $\mathcal{S}$ is infinite, so there are infinitely many gadgets in the graph. Consider a doubly-infinite path which visits gadgets corresponding to $E$ from first to last, and others from last to first. Correctness is clear.
\end{proof}

There is an easy reduction from directed to undirected Hamiltonicity in the classical case~\cite{Karp72}, and it goes through in the definable setting with no change. 

\begin{thm}\label{thm:undecidability-undirected-hamiltonicity}
\probref{UndirectedHamiltonicity} is undecidable.
\end{thm}
\begin{proof} 
By reduction from \probref{DirectedHamiltonicity}. Given a definable directed graph $G = (V,E)$, form an undirected graph $G'$ with three vertices $v_{in}, v_{mid}, v_{out}$ for each $v \in V$. For each $v \in V$ add edges $\{v_{in}, v_{mid}\}, \{v_{mid}, v_{out}\}$, and for each $(u,v) \in E$ add an edge $\{u_{out}, v_{in}\}$.
Then a Hamiltonian path in $G'$ must - up to reversal - always visit $v_{out}$ after $v_{mid}$ after $v_{in}$. Hence these correspond to directed Hamiltonian paths in $G$. The graph $G'$ is clearly definable and the construction is effective.
\end{proof}

We chose to consider doubly-infinite Hamiltonian paths, but essentially the same arguments apply to {\em singly}-infinite paths, which have a starting point and extend to infinity in one direction.
More precisely, the directed construction goes through exactly as is (although to make the argument easier we might like to add a new `initial' vertex with an edge to each end-vertex of each gadget). The reduction from the directed to the undirected does not go through as stated, but it can be fixed easily, by adding a pair of `initial vertices' $u,v$ with an edge between them and an edge from $v$ to $w_{in}$ for each $w$.

\subsection{3D matching}

3D matching is a generalisation of bipartite matching to 3-hypergraphs. Some classical formulations of this problem do not adapt to an infinite setting well, but the following one does.

\begin{prob}[\probname{3DMatching}]\label{prob:3DMatching}
\begin{description}
\item[Input] Definable sets $A, B, C$ and $R\subseteq A\times B\times C$
\item[Question] Is there a subset of $R$ whose projections onto $A,B,C$ are bijections?
\end{description}
\end{prob}

Karp in~\cite{Karp72} considered a special case where $A=B=C$, but our formulation easily reduces to that (hint: consider the disjoint union $A\uplus B\uplus C$), both in the finite and in the definable case. In~\cite{Karp72} hardness is proved by a reduction from \probref{ExactCover}. That reduction does not quite work in the definable setting, but we give one that works, based on the same general idea. 

\begin{thm}\label{thm:3dmatching}
\probref{3DMatching} is undecidable.
\end{thm}
\begin{proof} 
By reduction from~\probref{ExactTupleCover}. Given an instance $(X,\mathcal{S})$ (with $\mathcal{S}\subseteq X^{(n)}$ a set of non-repeating tuples), define:
\[
	Y = \{(\overline{x},i) \mid \overline{x}\in \mathcal{S}, 0\leq i<n\},
	\qquad A = Y\uplus \atoms, \qquad B = X\uplus\atoms.
\]
Then let $R \subseteq A \times A \times B$ contain three types of triples:
\begin{itemize}
    \item[(i)] $((\overline{x},i),(\overline{x},i),x_i)$, where $(\overline{x}, i) \in Y$ for $\overline{x} = (x_0, \dots, x_{n-1})$;
    \item[(ii)] $((\overline{x},i),(\overline{x},i+1),a)$, where $(\overline{x}, i) \in Y$, $a \in \atoms$, and the +1 is modulo $n$,
    \item[(iii)] all triples from $\atoms^3$.
\end{itemize}

Any $M\subseteq R$ is associated in an obvious way to a subset of the disjoint union $A\uplus A\uplus B$: an element of $A$ in the first component is included if it occurs on the first position in some triple in $M$, and so on for the second and third components. We say these are the elements $M$ \emph{covers}. Then $M$ is a 3D-matching if and only if every element of that disjoint union is covered exactly once (meaning that $M$ contains a unique triple covering each such element).

\begin{figure}[!htb]
\centering
\includegraphics[scale=1.0]{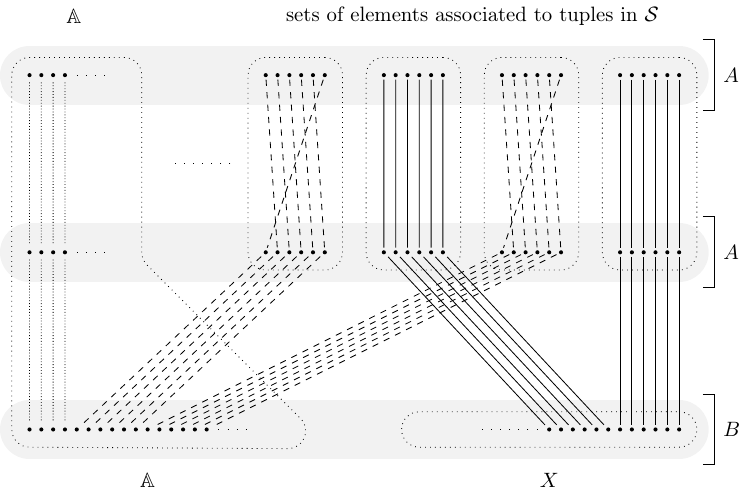}
\caption{Schematics of a matching from the proof of Theorem~\ref{thm:3dmatching}. Solid lines mark triples of type~(i), dashed lines -- of type~(ii), dotted lines -- of type~(iii).}
\label{fig:3d-matching-construction}
\end{figure}

Suppose $M \subseteq R$ is a 3D-matching. Then each $(\overline{x},i)\in A$ (from either of the first two components) can be covered by either a triple of type (i) or one of type (ii). The type (i) triple would cover $x_i$ in the third component as well, while the type (ii) triple would not. Importantly, if $(\overline{x},i)$ is covered in $M$ by a type (ii) triple, then all $(\overline{x},j)$ in both components must be covered in $M$ by type (ii) triples too. Hence for each $x \in X$ there is exactly one tuple $\overline{x}\in\mathcal{S}$ which contains $x$ and is associated to elements $(\overline{x},i)$ covered by triples of type (i). These tuples form an exact cover of $(X,\mathcal{S})$; see Figure~\ref{fig:3d-matching-construction} for an illustration.

Conversely, take an exact cover $E \subseteq \mathcal{S}$. For each $\overline{x} \in E$, include all the $n$ triples of type~(i) associated to it in the matching. For all the other tuples in $\mathcal{S}$, include a minimal set of triples of type (ii), each taking a fresh atom in $\atoms$ as the third component. This can be done so that there are infinitely many atoms left unused in $B$; choose any bijection $\sigma$ with all of $\atoms$ and include the type (iii) tuple $(\sigma a, \sigma a, a) \in \atoms^3$ for each such unused $a$. 
\end{proof}


\section{Decidability from extreme amenability}\label{sec:amenability}

We now turn attention to those problems from Karp's list that remain decidable in the definable setting. First we mildly generalise the technique used in~\cite{KKOT15}, used there to prove that finite-template Constraint Satisfaction Problems are decidable. This generalisation will let us prove decidability for a whole range of Karp's problems in one stroke. The few remaining problems, which are not covered by the technique, will be dealt with in Sections~\ref{sec:other}-\ref{sec:weighted}.

The development in this section starts from a general observation that in many search problems, a desired solution for an input structure $G$ can be seen as an expansion of $G$ to a relational structure over an extended vocabulary. For example, a $k$-coloring of a graph is a structure which (in addition to the edge relation of $G$ itself) includes $k$ unary predicates that specify the color of each vertex; a vertex cover (or any solution where a subset of vertices is sought for) adds to $G$ a single unary predicate; and so on. Of course, not every structure that extends a given $G$ is a legal solution; it is subject to conditions that depends on the problem in question. We will show that if the class of legal solutions for an equivariant structure has a certain closure property, then it contains a solution which is equivariant itself. Searching for equivariant solutions is often easy, which guarantees decidability.

\subsection{Compact sets of structures}\label{sec:compactness}

In what follows, we will need to briefly consider other atoms than the equality atoms. We will focus on relational structures $\atoms$ which are {\em oligomorphic}, i.e., such for every number $k$ the set $\atoms^{(k)}$ has finitely many orbits with respect to the canonical action of the automorphism group $\aut(\atoms)$. This is a standard notion in model theory; see e.g.~\cite{hodges} for a thorough introduction. The notion of definable and equivariant set from Section~\ref{sec:prelims} transports to this more general setting with little change; the only difference is that first-order formulas used in the defining set-builder expressions can use, apart from atom equality, relations from the vocabulary of $\atoms$. 

In our context, the most important example other than equality atoms is that of {\em ordered atoms} $(\mathbb{Q},\leq)$, i.e.~the total order of the rationals, whose automorphisms are order-preserving bijections. We will also consider extensions of this structure with finitely many constants: for $a_1,\ldots,a_n\in\mathbb{Q}$, the automorphisms of $(\mathbb{Q},\leq,a_1,\ldots,a_n)$ are those monotone bijections on $\mathbb{Q}$ which fix each $a_i$.

The following lemmas make sense over any oligomorphic atom structure $\atoms$.

Fix a finite relational signature $\sigma$ and a definable, equivariant set $X$. The set $\Sigma$ of all $\sigma$-structures on $X$ comes equipped with a topology, with a basis of neighbourhoods of the form:
\[
	\mathcal{B}_{Y}(M) = \{N\in\Sigma \mid N|_Y= M|_Y\}
\]
for $Y\subseteq X$ finite and $M\in\Sigma$, where $M|_Y$ is the induced substructure of $M$. 

\begin{lem}\label{lem:compact-continuous}
The space $\Sigma$ is compact, and the canonical action of $\aut(\atoms)$ on $\Sigma$ is continuous.
\end{lem}
\begin{proof}
The set $\Sigma$ is in bijection with $\prod_{R\in\sigma}\{0,1\}^{X^{\text{arity}(R)}}$,
and if $\sigma$ is finite then our topology is the product topology on this set, which is compact by Tychonoff's theorem. The action is continuous since $\mathcal{B}_{Y}(M)\cdot\pi = \mathcal{B}_{Y\cdot \pi}(M\cdot\pi)$ for every $\pi\in\aut(\atoms)$.
\end{proof}

For a structure $M\in\Sigma$ we write $\aage(M)$ to denote the $\aut(\atoms)$-closure of the set of finite induced substructures of $M$, and define a preorder $\sqsubseteq$ on $\Sigma$ by
\[M \sqsubseteq N \Longleftrightarrow \aage(M) \subseteq \aage(N).\]
For a structure $M$, the down-set $M\!\!\downarrow$ is the set of all $N$ such that $N\sqsubseteq M$.

\begin{lem}\label{lem:downset}
For every $M\in\Sigma$, the space ${M\!\!\downarrow}\subseteq\Sigma$ is equivariant and compact.
\end{lem}
\begin{proof}
Equivariance is obvious: if $M,N$ are in the same $\aut(\atoms)$-orbit of $\Sigma$ then they are $\sqsubseteq$-equivalent. For compactness, by Lemma~\ref{lem:compact-continuous} it is enough to show that $M\!\!\downarrow$ is closed. Take any $N\not\sqsubseteq M$. Then for some finite $Y$ the structure $N|_Y$ is not in $\aage(M)$ so  $N|_Y\neq M|_Y$ and the open neighbourhood $\mathcal{B}_Y(N)$ is disjoint from $M\!\!\downarrow$.
\end{proof}

The following result would apply over any (oligomorphic) atoms which are \emph{Ramsey} in the sense of~\cite{KPT05}. However, we will only need it for the ordered atoms with finitely many pairwise distinct constants, so this is how we formulate it.

\begin{thm}\label{thm:pestovery}
Over atoms $\atoms = (\mathbb{Q},\leq,a_1, \ldots, a_n)$, for any $X$ and $\sigma$ as above, every down-set in $\Sigma$ contains an equivariant structure.
\end{thm}
\begin{proof}
Consider any down-set ${M\!\!\downarrow}\subseteq\Sigma$. By Lemma~\ref{lem:downset} it is equivariant, so the continuous action from Lemma~\ref{lem:compact-continuous} restricts to it. By Pestov's theorem \cite{Pes98}, the group $\aut(\mathbb{Q},\leq)$ (considered as a topological group with the product topology) is {\em extremely amenable}: every continuous action of it on a nonempty compact space has a fixpoint. Fixpoints of the canonical action of $\aut(\atoms)$ are exactly equivariant elements. This proves the theorem for $n = 0$. For the general statement, note that $\aut(\atoms)$ is still extremely amenable; this follows from the isomorphism of topological groups $\aut(\atoms) \cong (\aut(\mathbb{Q},<))^{n+1}$, since the class of extremely amenable groups is closed under direct products (see \cite[Lem.~6.7(iii)]{KPT05}).
\end{proof}

\subsection{Decidability}

Our decidability proofs in this section will follow the pattern used in~\cite{KKOT15}. Given a problem instance $S$-definable over equality atoms, we:
\begin{itemize}
\item[(i)] See it as an equivariant structure over ordered atoms with constants from $S$;
\item[(ii)] Understand solutions to that instance as structures over a finite relational signature;
\item[(iii)] Show that if the set of legal solutions is non-empty then it contains a down-set in the sense of Sec.~\ref{sec:compactness};
\item[(iv)] Infer from Theorem~\ref{thm:pestovery} that if a solution exists then an equivariant one exists;
\item[(v)] Show that looking for an equivariant solution is a decidable problem.
\end{itemize}
In practice, a convenient way to ensure (iii) is to prove that the set of legal solutions is downwards-closed with respect to $\sqsubseteq$. In particular, this happens if the set of legal solutions is described by a (possibly infinite) family of finite forbidden structures, i.e., a family $(H_i)_{i\in I}$ of finite $\sigma$-structures such that a solution $M$ is legal if and only if $H_i\not\in\aage(M)$ for every $i\in I$.

As a first example, consider the restriction of \probref{CNFSat} to formulas with finite clauses. The decidability of this problem follows from~\cite[Thm.~3]{KKOT15}, but we reprove it here as an illustrative application of Theorem~\ref{thm:pestovery}.

\begin{prob}[\probname{3-Sat}]\label{prob:3-Sat}
\begin{description}
\item[Input] Definable CNF formula $\varphi$ with at most $3$ literals per clause
\item[Question] Is $\varphi$ satisfiable?
\end{description}
\end{prob}

\begin{thm}\label{thm:decidability-3-sat}
\probref{3-Sat} is decidable.
\end{thm}
\begin{proof}
Take a $3$-CNF formula $\varphi$, $S$-definable over equality atoms, over an $S$-definable set $X$ of variables, for $S\subseteq\atoms$. (i) If we impose on $\atoms$ any ordering isomorphic to the total order of the rationals, $\varphi$ remains $S$-definable over ordered atoms. Adding the atoms from $S$ to $\atoms$ as constants, $\varphi$ becomes equivariant. 

(ii) A valuation for $\varphi$ (satisfying or not) can be seen as a predicate on $X$, i.e.~a structure over a signature $\sigma$ with a single predicate symbol. (iii) The set of satisfying valuations is described by an equivariant set of forbidden finite substructures, namely the partial valuations which invalidate one of the clauses of $\varphi$. As a result, the set of satisfying valuations for $\varphi$ is downwards-closed, so (iv) by Theorem~\ref{thm:pestovery} it either is empty or contains an equivariant valuation.

To check if $\varphi$ is satisfiable, it is therefore enough to look for a satisfying valuation which is equivariant over our extended atoms or, equivalently, $S$-definable over ordered atoms. (v) This is easy to do: there are only finitely many $S$-orbits of variables in $X$, and an $S$-definable valuation must be constant in every $S$-orbit, so there are finitely many valuations to consider, and for each of them it is easy to decide whether it satisfies every clause in $\varphi$.
\end{proof}
 
The same argument shows that $k$\probname{-Sat} is decidable for any $k$. It also means that \probref{CNFSat} is decidable when restricted to formulas containing only finite clauses. This is because every such formula $\varphi$, having an orbit-finite set of clauses, has an upper bound on the size of a clause, so it is an instance of $k$\probname{-Sat} for some $k$ (which can be computed from $\varphi$).
 
It may be illustrative to see where the argument fails for the unrestricted \probref{CNFSat}. Consider an equivariant CNF-formula with two clauses over $X=\{x_a\mid a\in\atoms\}$:
\[\textstyle
	\varphi = \left(\bigvee_{a\in\atoms}x_a\right) \land \left(\bigvee_{a\in\atoms}\neg x_a\right).
\]
Consider any satisfying valuation, understood as a structure $M$ on $X$ over a single predicate symbol. Assume that in the valuation infinitely many variables are false. (This does not lose generality: the symmetric case is that infinitely many variables are {\em true}.) Then $\aage(M)$ contains (among other things) all finite substructures where no element satisfies the predicate. Let $N$ be a structure on $X$ where no element satisfies the predicate. Then $N\sqsubseteq M$. But the valuation corresponding to $N$ does not satisfy $\varphi$, so the set of satisfying valuations does not contain $M\!\!\downarrow$ and Theorem~\ref{thm:pestovery} does not apply. Indeed, even though $\varphi$ has plenty of satisfying valuations, none of them are equivariant.

The same machinery directly applies to a few more problems. In all these, the input is a definable graph $G$.

\begin{prob}[\probname{VertexCover}]\label{prob:VertexCover}
\begin{description}
\item[Question] Does $G$ have a finite vertex cover?
\end{description}
\end{prob}

\begin{prob}[\probname{FeedbackVertexSet}]\label{prob:FeedbackVertexSet}
\begin{description}
\item[Question] Is there a finite set of vertices whose removal makes $G$ acyclic?
\end{description}
\end{prob}

\begin{prob}[\probname{FeedbackArcSet}]\label{prob:FeedbackArcSet}
\begin{description}
\item[Question] Is there a finite set of edges whose removal makes $G$ acyclic?
\end{description}
\end{prob}

In the finite setting of~\cite{Karp72}, the input to these problems includes a number $k$, and one asks whether a vertex cover (etc.) of size at most $k$ exists. We could do the same here, and we call these variants (defined in the obvious way) $k$-\probref{VertexCover}, $k$-\probref{FeedbackVertexSet}, and $k$-\probref{FeedbackArcSet}.

These variants, however, are decidable for very easy reasons: for a fixed $k$, there are only orbit-finitely many sets of vertices/edges of size $k$, and they can be effectively enumerated in the search for a solution. The non `$k$-' versions, as formulated above, are more interesting, but nevertheless:

\begin{thm}\label{thm:vertex-cover}
\probref{VertexCover}, \probref{FeedbackVertexSet} and \probref{FeedbackArcSet} are decidable.
\end{thm} 
\begin{proof}
We proceed as for Thm.~\ref{thm:decidability-3-sat}, arguing only for part (iii). Consider \probref{VertexCover}. A choice of vertices in $G$ can be seen as a predicate $M$ on the set of all vertices. If $M$ is a finite vertex cover (of size, say, $k$) then $\aage(M)$ can be described by a set of forbidden finite structures: (a) those where more than $k$ vertices satisfy the predicate, and (b) those where neither end of some edge satisfies the predicate. This implies that $M\!\!\downarrow$ contains only finite vertex covers, and Thm.~\ref{thm:pestovery} applies.

For~\probref{FeedbackVertexSet}, the argument is very similar, except that the forbidden finite structures are (a) those where more than $k$ vertices satisfy the predicate, and (b) those with some cycle that does not visit any vertex that satisfies the predicate. Note that in this case there is no way to describe $\aage(M)$ with an {\em orbit-finite} set of forbidden structures. This, however, is not needed for showing that the set of legal solutions is downwards-closed.

For~\probref{FeedbackArcSet}, the situation is similar, except that a solution is a choice of edges rather than vertices. One way to represent this is with a binary relation $R$ on the set of all vertices. The set of forbidden finite structures that describes a given solution (of size $k$) is: (a) those where some two vertices related by $R$ are not connected by an edge, (b) those where $R$ contains more than $k$ pairs, and (c) those with some cycle that does not visit any edge in $R$.
\end{proof}

\begin{prob}[\probname{Colorability}]\label{prob:Colorability}
\begin{description}
\item[Question] Does $G$ have a vertex coloring with finitely many colors?
\end{description}
\end{prob}

\begin{prob}[\probname{CliqueCover}]\label{prob:CliqueCover}
\begin{description}
\item[Question] Is $V$ a disjoint union of finitely many cliques (where $G = (V,E)$)?
\end{description}
\end{prob}

These two problems are equivalent (replace $G$ with its complement to reduce one to the other), so let us focus on colorability. As for Problems~\ref{prob:VertexCover}-\ref{prob:FeedbackArcSet}, in the finite setting of~\cite{Karp72} the input includes a number $k$ and a $k$-coloring is sought for. Unlike for Problems~\ref{prob:VertexCover}-\ref{prob:FeedbackArcSet}, that problem remains interesting in our setting. Decidability of $k$-\probref{Colorability} follows from~\cite[Thm.~3]{KKOT15}, but it is easy to give an argument analogous to Thm.~\ref{thm:decidability-3-sat}. Here, $k$-colorings can be seen as structures over a signature with $k$ predicate symbols, and the set of legal $k$-colorings for any fixed $G$ is downwards-closed.

Unrestricted colorability easily follows from this:
\begin{thm}\label{thm:colorability}
\probref{Colorability} (therefore also \probref{CliqueCover}) is decidable.
\end{thm} 
\begin{proof}
Every finite coloring is a $k$-coloring for some $k$, so if one exists then an equivariant $k$-coloring exists. An equivariant coloring must be constant on every orbit of vertices, so it cannot use more colors than there are orbits. The number of orbits can be computed from $G$, and the set of candidate equivariant colorings can be enumerated and checked for legality.
\end{proof}


\section{Other decidable cases}\label{sec:other}

The key feature of the problems from Section~\ref{sec:amenability} was that, whenever they admit a solution to an $S$-definable instance, then an $S$-definable solution also exists. Some problems do not have this property, for example:
\begin{prob}[\probname{Clique}]\label{prob:Clique}
\begin{description}
\item[Input] Definable graph $G$
\item[Question] Does $G$ contain an infinite clique?
\end{description}
\end{prob}
Consider for instance a graph with ordered pairs from $\atoms^{(2)}$ as vertices, with edges between exactly those vertices that share the first component. This equivariant graph contains an infinite clique, indeed it is a disjoint union of infinite cliques, but none of these cliques are equivariant.

Nevertheless, the problem is decidable. This has already been proved for automatic graphs~\cite{KL10,Rub08}, but for definable graphs a more direct argument exists. We will use the following elementary property of equality atoms; here and in the following $v_1\sim_S v_2$ means that $v_1$ and $v_2$ are in the same $S$-orbit.

\begin{lem}\label{lem:basic-lemma}
Consider any finite $S\subseteq\atoms$, elements $v_1\sim_S v_2$ and an element $w$ such that
\[
	\supp(v_1)\cap\supp(w)\subseteq S \qquad \text{and}\qquad \supp(v_2)\cap\supp(w)\subseteq S.
\]
Then $\{v_1,w\}\sim_S\{v_2,w\}$.
\end{lem}
\begin{proof}
Let $\pi$ be an $S$-automorphism such that $v_1\cdot\pi =v_2$. The value of $v_1\cdot\pi$ only depends on the action of $\pi$ on $\supp(v_1)$. The above assumptions imply that the set $\supp(w)\setminus S$ is disjoint from $\supp(v_1)$ and $\supp(v_2)$, so we can find an $S$-automorphism $\pi'$ that agrees with $\pi$ on $\supp(v_1)$ but fixes $\supp(w)\setminus S$ pointwise. Then $v_1\cdot\pi' =v_2$ and $w\cdot\pi'=w$ as required.
\end{proof}

\begin{thm}\label{thm:clique}
\probref{Clique} is decidable.
\end{thm} 
\begin{proof}
A definable graph $G$ is orbit-finite, so any infinite clique in it must have an infinite intersection with one of the orbits. We may therefore focus on graphs with one orbit of vertices. Further, in an infinite clique, every edge $\{v_1,v_2\}$ can be colored with its orbit. There are finitely many orbits of edges, so by Ramsey's theorem the clique contains an infinite sub-clique where every edge is in the same orbit. It is therefore enough to decide, for a given orbit of edges in $G$, whether $G$ contains infinitely many vertices whose every pair belongs to that orbit.

Over equality atoms, this is easy to do: the condition holds for the orbit of $\{v_1,v_2\}$ if and only if $v_1\sim_S v_2$, where $S=\supp(v_1)\cap\supp(v_2)$. (This condition is easy to decide by looking at $v_1$ and $v_2$.)

To see this, first assume $v_1\sim_S v_2$. Since $v_1\neq v_2$, the set $\supp(v_1)$ (and $\supp(v_2)$) must contain some atoms outside of $S$. Replacing them with suitably chosen fresh atoms, we can find infinitely many pairwise different vertices $v_1,v_2,v_3\ldots$, all in the same $S$-orbit, so that $\supp(v_i)\cap\supp(v_j)=S$ for all $i\neq j$. For any four pairwise distinct numbers $i,j,i',j'$, applying Lemma~\ref{lem:basic-lemma} twice we get
\[
	\{v_i,v_j\} \sim_S \{v_{i'},v_j\} \sim_S \{v_{i'},v_{j'}\}.
\] 
This implies that all pairs $\{v_i,v_j\}$ for $i\neq j$ are in the same $S$-orbit, hence in the same orbit.

In the other direction, assume an infinite set $v_1,v_2,\ldots$ of vertices such that the pairs $\{v_i,v_j\}$ are in the same orbit for all $i\neq j$. The size of $\supp(v_i)$ and $S_{ij} = \supp(v_i)\cap\supp(v_j)$ does not depend on $i$ and $j$, so by basic combinatorics (specifically, by Deza's sunflower lemma~\cite{deza74}) the intersection $S_{ij}$ itself does not depend on $i$ and $j$; call this shared intersection $S$. For any fixed $i\neq j$, since $v_i$ and $v_j$ are in the same orbit, there is an atom automorphims $\pi_{ij}$ that preserves $S$ set-wise and such that $v_i\cdot\pi_{ij}=v_j$. 
Keeping $i$ fixed and choosing more than $|S|!$ different $j$'s, there are some $j\neq j'$ such that $\pi_{ij}$ and $\pi_{ij'}$ agree on $S$. Then the composition of $\pi_{ij}^{-1}$ and $\pi_{ij'}$ maps $v_j$ to $v_{j'}$ and fixes $S$ pointwise, hence $v_j\sim_S v_{j'}$. But the property ``$v\sim_S w$ for $S=\supp(v)\cap\supp(w)$'' is preserved by arbitrary atom automorphisms, so -- since all $\{v_i,v_j\}$ are in the same orbit -- we obtain $v_i\sim_S v_j$ for all $i\neq j$.
\end{proof}

The following was proved decidable for automatic structures in~\cite{Koc14} with a direct argument, rather than by an immediate reduction to~\probref{Clique}:

\begin{prob}[\probname{SetPacking}]\label{prob:SetPacking}
\begin{description}
\item[Input] Definable set $X$, definable family $\mathcal{S}$ of subsets of $X$
\item[Question] Does $\mathcal{S}$ contain an infinite, pairwise disjoint subfamily?
\end{description}
\end{prob}
\begin{thm}\label{thm:setpacking-decidability}
\probref{SetPacking} is decidable.
\end{thm} 
\begin{proof}
Given $X$ and $\mathcal{S}$, solve \probref{Clique} for a graph with $\mathcal{S}$ as the set of vertices, and an edge between two sets if and only if they are disjoint.
\end{proof}

Here is another problem that does not fall into the scope of Section~\ref{sec:amenability}:
\begin{prob}[\probname{SetCovering}]\label{prob:SetCovering}
\vspace{-\topsep}
\begin{description}
\item[Input] Definable set $X$, definable family $\mathcal{S}$ of subsets of~$X$
\item[Question] Does $\mathcal{S}$ contain a finite subfamily $C$ whose union is $X$?
\end{description}
\end{prob}
In~\cite{KL10}, a variant of this problem where $C$ is required to be co-infinite (i.e. such that $\mathcal{S}\setminus C$ is infinite) rather than finite, was proved decidable for all automatic structures. The technique used there does not seem applicable to our formulation, but at least on definable instances the following easy argument works.
\begin{thm}
\probref{SetCovering} is decidable. \label{thm:setcovering-decidable}
\end{thm} 
\begin{proof}
Clearly we can reduce to the case where $X$ is a single infinite orbit. Assume that a finite family $C\subseteq \mathcal{S}$ covers $X$. Every set in $C$ has some finite support, and the union of these supports for all sets in $C$ is also finite, so there must be some $x\in Y\in C$ such that $\supp(x)$ and $\supp(Y)$ are disjoint modulo $\supp(X,\mathcal{S})$ (by which we mean their intersection lies in $\supp(X,\mathcal{S})$; this is empty if the instance is equivariant).

On the other hand, assume that such $x\in Y\in \mathcal{S}$ exist. Then $Y$ contains {\em all} $z\in X$ such that $\supp(z)$ and $\supp(Y)$ are disjoint (again, and in the following, this is modulo $\supp(X,\mathcal{S})$). Take $Y_1,Y_2\ldots,Y_{|\supp(x)|+1}$ in the orbit of $Y$ with pairwise disjoint supports. Then every element of $X$ has support disjoint from some of the $Y_i$, so it belongs to $Y_i$. Hence all the $Y_i$'s jointly cover $X$.

Finally, it is easy to effectively search for $x\in Y\in \mathcal{S}$ as above.
\end{proof}

This gives an alternative decidability proof of \probref{VertexCover} (Problem~\ref{prob:VertexCover}), via a straightforward reduction used already by Karp~\cite{Karp72}. The argument in Theorem~\ref{thm:vertex-cover} is still worth making though, as it exhibits additional structure in the space of vertex covers that is missing in the more general case of set coverings.


\vspace{-1ex}

\section{Weighted problems}\label{sec:weighted}

A few problems on Karp's list involve adding up sets of numbers, be it weights of graph edges, penalties in job sequencing etc. This poses an obvious difficulty in generalising these problems to infinite structures. One may try to follow the idea that we used in Sections~\ref{sec:amenability}-\ref{sec:other} and, rather than comparing various quantities to a fixed input number $k$, require them to be finite (or infinite). In the process, however, the essence of a weighted problem usually seems to be lost.

Consider, for example, the problem of finding a minimal Steiner tree in a weighted graph. In the finite formulation~\cite{Karp72}, given an edge-weighted graph $G$, a subset $W$ of its vertices and a number $k$, one asks whether $G$ has a tree of total weight at most $k$ that spans all the vertices in $W$. In an infinite setting, for a ``total weight'' to make sense, one has to restrict attention to graphs with non-negative weights. Then one could ask whether a spanning tree with a {\em finite} total weight exists. (The problem with a given bound $k$ is easily decidable for the same reason as Problems~\ref{prob:VertexCover}-\ref{prob:FeedbackArcSet}.)
This is a valid  question, but the values of weights are lost in it: it is equivalent to asking whether a spanning tree exists that uses only finitely many edges with non-zero weights. The problem then simplifies to:

\begin{prob}[\probname{SteinerTree}]\label{prob:SteinerTree}
\begin{description}
\item[Input] Definable graph $G=(V,E)$, definable subsets $W\subseteq V$ and $F\subseteq E$
\item[Question] Does $G$ have a tree that spans all vertices from $W$ and uses finitely many edges from $F$?
\end{description}
\end{prob}

\begin{thm}
\probref{SteinerTree} is decidable. \label{thm:steinertree-decidable}
\end{thm} 
\begin{proof}
One needs to check if ($G$ is connected and) the vertices from $W$ fall into finitely many connected components of the graph $(V,E\setminus F)$. To this end, compute its transitive closure (such fixpoint calculations are effective; see~\cite{kl19} for a general study or~\cite{ks16} for an implementation), restrict to the connected components with vertices from $W$, and solve \probref{CliqueCover} (Problem~\ref{prob:CliqueCover}).
\end{proof}
Admittedly, the essence of the original weighted problem is lost to some extent in this formulation. For other problems, it gets worse. Karp's problem of finding a maximal cut in a weighted graph becomes:

\begin{prob}[\probname{MaxCut}]\label{prob:MaxCut}
\begin{description}
\item[Input] Definable graph $G=(V,E)$, definable subset $F\subseteq E$ of edges
\item[Question] Is there a subset $W\subseteq V$ such that $F$ contains infinitely many edges between $W$ and $V\setminus W$?
\end{description}
\end{prob}
Here even the set $E$ becomes irrelevant. It is easy to see that the condition holds if and only if $F$ is infinite, which is obviously decidable.

The remaining three problems on Karp's list~\cite{Karp72} are: knapsack, job sequencing and number partitioning. These seem even less open to infinite generalisations than the two above, so we leave them untreated.


\section{Complexity considerations}\label{sec:complexity}

Our main focus has been on decidability, but one naturally wonders also about the complexity of the decidable infinite versions of Karp's problems. The picture, unfortunately, is a little more fiddly.

The first difficulty is that it begins to matter how we choose to represent definable sets, and the choice is not obvious. In Section~\ref{sec:prelims} we decided to use \emph{set-builder expressions}: each definable set is represented by such an expression together with a valuation on its free variables. The advantage of this representation is that many basic constructions on definable sets can be produced in polynomial time (see \cite[Chap.~4]{atom-book}). However, some very basic operations are still computationally difficult: even checking emptiness of a given set is \PSPACE-complete, since it amounts - over the equality atoms - to deciding the first order theory of equality\footnote{\ We could restrict our expressions to use only quantifier-free formulas, but checking equality of definable sets would still be \PSPACE-complete - see Exercise 160 in \cite{atom-book}.}.

Moreover, it might be argued that set-builder expressions do not, from a complexity perspective, extend quite the right classical problems: when used to represent finite structures, they can sometimes be exponentially more compact than traditional representations. In a sense, as evidenced by the proof of Theorem~\ref{thm:nexphard} below, they correspond better to \emph{succinct graph problems}, as introduced by Galperin and Widgerson \cite{Galp84}, than to Karp's originals. 

For an alternative representation, one could postulate that a definable set is either an atom, represented as itself, or a set, represented by a finite set of atoms (the support) and a list of representatives of orbits. These representatives are themselves definable sets, each of which is represented using the same scheme.

In the finite case, this encoding essentially agrees with classical ones. However, it causes other unpleasant problems: even computing binary Cartesian products can result in an exponential blow-up (see~\cite[Sec.~10.3]{BKL14}), and modest sets such as $\atoms^n$ have representations exponentially long with respect to $n$, which may make computations on them look deceptively efficient. 

For consistency, we continue to represent definable sets with set-builder expressions as in Section~\ref{sec:prelims}, with the understanding that a different representation scheme could result in a different complexity landscape. Moreover, we should specify a few details of this representation, which did not matter before, but regrettably do now.

Firstly, we intend only to use the basic language introduced in Section~\ref{sec:prelims}; the syntactic sugar mentioned there will live only on the page, for expositional purposes.
Secondly, we will assume that the representation of $(a_1, \ldots, a_n)$ has size linear in the sum of the sizes of the chosen representatives of $a_1, \ldots, a_n$. If pairs are represented by Kuratowski pairs and tuples as nested pairs, this just means that we assume right-associative bracketing: for instance, we bracket $(a_1, a_2, a_3)$ as $(a_1, (a_2, a_3))$. Another choice could cause an unwanted exponential blow-up.

\subsection{Upper bounds}

All our problems will be either \PSPACE-complete or \NEXP-complete. We start with upper bounds, using the algorithms described before (or minor variants thereof). Before that, however, we should review the complexity of a few basic operations. Rather than providing a general but involved statement (see~\cite[Chap.~10]{atom-book} for that), we select  operations which in combination will be sufficiently powerful, yet relatively easy to work with. The following results apply over both equality and ordered atoms.

Some elementary computations on definable sets are obvious. For instance, in linear time we can compute unions of sets,  tell the difference between atoms and sets, and find finite supports (not necessarily least supports) of given sets. For the latter, it is enough to list all the atoms syntactically present in the defining expression.

\begin{lem}\label{lem:pspace-basics} There are \PSPACE algorithms which, given definable sets $A,B$ and a finite set $S\subseteq\atoms$:
\begin{enumerate}[(a)]
    \item check whether $A$ is empty;
    \item decide whether $A = B$, $A\in B$, $A\subseteq B$;
    \item return the $S$-orbit of $A$;
    \item list\,\footnote{\ The number of $S$-orbits in $A$ may be exponential in the size of $(A,S)$. We assume, therefore, that a suitable model of computation is used, where the algorithm outputs the representatives one at a time on a designated tape. In particular, the outputs should individually be of polynomial size.} the $S$-orbits in $A$, assuming that $S$ supports $A$;
    \item check whether $S$ supports $A$;
    \item return the least support $\supp(A)$.
\end{enumerate} \label{lem:basic-complexity-results}
\begin{proof} 
For (a), the only interesting case is that of a single set-builder expression
\begin{align}\label{eq:setexp}
	\{e(\bar{x},\bar{y}) \mid \bar{x}\in\atoms, \varphi(\bar{x},\bar{y})\}[\bar{a}]
\end{align}
where $\overline{y}$ (which we assume disjoint from $\overline{x}$) lists the free variables; $\overline{a}$ is the tuple of atoms assigned to them.
 Here checking emptiness amounts to checking satisfiability for the first-order formula $\varphi(\bar{x},\bar{a})$, which can be done in \PSPACE both for equality and ordered atoms.

For (b), the three algorithms are defined by mutual recursion on the syntactic structure of expressions. For example, to check whether
\[
	e[\bar{a}] \in \{e'(\bar{x},\bar{y}) \mid \bar{x}\in\atoms, \varphi(\bar{x},\bar{y})\}[\bar{b}],
\]
one needs to check if there is some tuple of atoms $\bar{c}$, of length the same as $\bar{x}$, such that $\varphi(\bar{c},\bar{b})$ holds and $e[\bar{a}]=e'[\bar{c},\bar{b}]$. This only depends on the $\bar{a}\bar{b}$-orbit of $\bar{c}$, or equivalently on the \emph{order type}\footnote{\ Meaning the equalities and order relations of the atoms in $\bar{c}$ with those in $\bar{a}$ and $\bar{b}$.} of $\bar{c}$ over $\bar{a}\bar{b}$, so there are only finitely many essentially different candidates for $\bar{c}$ and they can be exhaustively searched. For other cases and more details, see~\cite[Thm.~10.11]{atom-book}.

For (c), note the $S$-orbit of $A = e[\overline{a}]$ is exactly the set of all $e[\overline{b}]$ such that $\overline{b}$ has the same order type over $S$ as $\overline{a}$. But we may write a formula (using fresh variables assigned to elements of $S$) which expresses this, and thus describe this orbit by a set-builder expression.

For (d), the interesting case is $A$ represented by a single set-builder expression as in~\eqref{eq:setexp}. 
For any element $e[\overline{b},\overline{a}]$, its $S$-orbit consists of all $e[\overline{c},\overline{a}]$ where $\overline{c},\overline{a}$ has the same order type over $S$ as $\overline{b},\overline{a}$. Again, we may write a formula which, when adjoined to $\phi(\overline{x},\overline{y})$, expresses this, and thereby obtain a representative of this orbit.
To list $S$-orbits, then, we need to list order-types over $S$, and for each adjoin a formula as above. Because we are starting from the order-type, we do need to check for emptiness, but by (a) this is not an issue. We may also list the same orbit twice, but the problem statement allows this.

For (e), use (c) and (b) to check if the $S$-orbit of $A$ is equal to $\{A\}$.

For (f), first let $S$ contain all the atoms present in the defining expression for $A$; this is an overapproximation of $\supp(S)$. Then, for each subset of $S$ we may use (e) to check if it supports $A$.
\end{proof}
\end{lem}

The following is both a useful lemma and an example of how the operations above may be combined:

\begin{lem} We can determine in \PSPACE whether a graph $G$ admits a path from a vertex $u$ to a vertex $v$. \label{lem:reachability-pspace}
\begin{proof} Let $S=\supp(G,u,v)$, and note that if there is such a path then there is one in which no two vertices lie in the same $S$-orbit. In \PSPACE we can compute $S$ and then count the $S$-orbits of vertices. The idea is then to proceed by guessing a path of length at most this number and beginning at $u$, accepting if we reach $v$.

At any stage we must simply keep track of how many edges we have traversed and which vertex $w$ we are at: to guess a next vertex, we compute $S'=\supp(G,u,v,w)$, list the $S'$-orbits of edges in $G$, guess one, and take its representative $(x,y)$. After verifying that $x = w$ (reject otherwise) we take $y$ as our new vertex. This procedure runs in \NPSPACE, hence the problem is in \PSPACE by Savitch's Theorem \cite{savitch70}.

One subtlety here is that at each stage, we do not exactly guess a next vertex, but instead guess an orbit of vertices under an appropriate support, and take one from that orbit deterministically. As such, we are not guaranteed to be able to guess any given (short) path $P$ from $u$ to $v$. However, we can guess a path such that the $S$-orbit of the $i$th and $(i+1)$st vertices is the same as that of $(P_i, P_{i+1})$ for all $i$, and this will do.
\end{proof}
\end{lem}

We are now ready to study algorithms for Karp's problems. We begin with \probref{3-Sat}. Recall from Section \ref{sec:amenability} that if an $S$-supported \probref{3-Sat} formula  has a satisfying valuation then it has an $S$-supported one (over the ordered atoms). The algorithm suggested by this is to enumerate all such valuations, testing each in turn. We claim this is in \NEXP.

\begin{thm} \probref{3-Sat} $\in$ \NEXP. \label{thm:3sat-in-nexp}
\end{thm}
\begin{proof} We work over the ordered atoms. Given $\varphi$, take a support $S$ and list the (exponentially many) $S$-orbits of variables in $\varphi$. 
Then nondeterministically guess a subset of these orbits, and seek to verify that these form a satisfying valuation. To do so, list representatives of an orbit of clauses. For each of those, extract its literals, test whether they any of them is set to true 
and thereby determine whether the clause (and hence any other in its orbit) is satisfied. This process also takes exponential time.
\end{proof}

Similar analysis would apply (with a few details changed) to \probref{VertexCover}, \probref{FeedbackVertexSet}, \probref{FeedbackArcSet}, \probref{Colorability}, and \probref{CliqueCover}. However, for these problems we can do better, because we do not need to remember an exponentially long solution:

\begin{thm} \probref{VertexCover}, \probref{FeedbackVertexSet}, \probref{FeedbackArcSet}, \probref{Colorability}, \probref{CliqueCover} $\in$ \PSPACE. \label{thm:many-pspace}
\end{thm}
\begin{proof} Let us start with \probref{VertexCover}. Consider a definable graph $G$ over the ordered atoms, with support $S$. We know from Section~\ref{sec:amenability} that if $G$ has a finite vertex cover then it has an $S$-supported one. But then each vertex in that cover is $S$-supported, so the set of all (finitely many) $S$-supported vertices is a vertex cover. Hence, to test whether $G$ admits a finite vertex cover, it is enough to verify that every edge has an $S$-supported end-vertex, which can be done in \PSPACE.

Similarly, if we instead ask for a finite feedback vertex set, the only candidate we care about is the set of all $S$-supported elements. To check if this is a feedback vertex set, however, we need to check that there is no cycle using only non-$S$-supported elements. This is easy to do in \PSPACE along the same lines as in Lemma \ref{lem:reachability-pspace}. The case of \probref{FeedbackArcSet} is similar, except that we are dealing with edges rather than vertices.

\probref{Colorability} reduces to checking that there is no edge with both end-vertices in the same $S$-orbit. This is easy to do in \PSPACE. \probref{CliqueCover} is, as before, equivalent to \probref{Colorability}.
\end{proof}

A few more problems remain, which were shown decidable in Section~\ref{sec:other} by ad-hoc methods. We treat them individually:

\begin{thm} \probref{Clique} $\in$ \PSPACE.
\end{thm}
\begin{proof} We showed in Theorem~\ref{thm:clique} that a definable, undirected, equivariant graph $G$ admits an infinite clique if and only if it contains an edge $\{u,v\}$ with $u,v$ in the same $\supp(u) \cap \supp(v)$-orbit. For a given edge, this can certainly be checked in \PSPACE, and so we can simply list representatives of orbits of edges, checking the condition for each in turn.
\end{proof}

\begin{thm} \probref{SetPacking} $\in$ \PSPACE.
\end{thm}
\begin{proof} Run the algorithm above on the graph constructed in Theorem \ref{thm:setpacking-decidability}. We do not need to write down the set of edges at any point, since we can simply list representatives of orbits of pairs of vertices, and check for each in turn whether they should be edges (ignoring them if not).
\end{proof}

\begin{thm} \probref{SetCovering} $\in$ \PSPACE.
\end{thm}
\begin{proof} In Theorem \ref{thm:setcovering-decidable}, we show that an instance $(X,\mathcal{S})$ with least support $S$ admits a finite set covering exactly if for each infinite $S$-orbit $\mathcal{O}$ in $X$, there is $Y \in \mathcal{S}$ and $x \in Y \cap \mathcal{O}$ such that $\supp(x) \cap \supp(Y) \subseteq S$.

We may therefore compute $S=\supp(X,\mathcal{S})$, and for each $S$-orbit $\mathcal{O}$ in $X$ check whether it is infinite (i.e.~contains at least two elements). The set of pairs $(x,Y)$ such that $x\in Y\in\mathcal{S}$ is definable, so we can list (the representatives of) its $S$-orbits. For each such pair we can check in \PSPACE that $\supp(x) \cap \supp(Y) \subseteq S$ and that the $S$-orbit of $x$ is infinite.
\end{proof}

Finally, we turn to the weighted problems, which we did not think translated very naturally to the infinite setting. The first was \probref{SteinerTree}:

\begin{thm} \probref{SteinerTree} $\in$ \PSPACE.
\end{thm}
\begin{proof} The input consists of a graph $G = (V,E)$, a set $W \subseteq V$ and a set $F \subseteq E$. We argued in Theorem \ref{thm:steinertree-decidable} that we must simply check that $G$ is connected and that the vertices of $W$ fall into finitely many connected components of $(V,E \setminus F)$.

To check that $G$ is connected, we can take representatives of pairs of vertices and look for a path joining each via Lemma \ref{lem:reachability-pspace}. To check that the vertices of $W$ fall into finitely many connected components of $(V,E \setminus F)$, take a joint support $S$ of everything in the input, and note that it suffices to check that each $S$-orbit of vertices in $W$ (every $S$-orbit of vertices is a subset of $W$ or $V \setminus W$) lies in a single connected component of $(V,E \setminus F)$, which we can verify in \PSPACE along the lines of Lemma \ref{lem:reachability-pspace}.

To see that this suffices, consider an $S$-orbit $\mathcal{O}$, and the orbit of pairs $(u,v) \in \mathcal{O}^2$ whose supports are disjoint modulo $S$, i.e. with $\supp(u) \cap \supp(v) \subseteq S$. The transitive closure of $(V, E \setminus F)$ is $S$-supported, so either all these pairs are edges in that graph or none of them are. In the latter case, clearly we have infinitely many connected components, and in the former, take any $u,v \in \mathcal{O}$ and $w \in \mathcal{O}$ fresh with respect to both of them: we see $(u,w)$ and $(w,v)$ are edges in the transitive closure, so $u,v$ lie in the same connected component. 
\end{proof}

And lastly, we have \probref{MaxCut}, which translated to the problem of whether one of the input sets was infinite. For completeness, however:

\begin{thm} \probref{MaxCut} $\in$ \PSPACE.
\end{thm}
\begin{proof} We must simply check whether $F$, the set of edges in the input, is infinite, i.e. contains a $\supp(F)$-orbit of size at least two. This may be done in \PSPACE.
\end{proof}

An interesting distinction emerges here with the `$k$-' variants of the above problems. We dismissed most of these earlier as uninteresting, since (with the exception of $k$-\probref{Colorability} and $k$-\probref{CliqueCover}) their decidability is immediate from basic principles. 
From a complexity perspective, however, the story is a little different. Take $k$-\probref{VertexCover} as an example. We have membership in \NEXP by the following obvious algorithm:

\begin{thm}\label{thm:pspace-last} $k$-\probref{VertexCover} $\in$ \NEXP.
\end{thm}
\begin{proof} Let $G$ be an equivariant graph, and let $d$ be the maximum size of the least support of a vertex of $G$ (also known as the dimension of $G$). Let $S$ be a set of $kd$ fresh atoms. Note that $G$ contains a vertex cover of size $k$ exactly if it contains an $S$-supported one.

Yet we may list the $S$-supported elements in $G$ in exponential time\footnote{\ To see this from Lemma \ref{lem:basic-complexity-results}, note each $S$-supported element has some support $T \subseteq S$ of size at most $d$. In particular, we may iterate through subsets of $S$ of size $d$ (there are at most exponentially many in the size of the input), and for each list the elements it supports in exponential time, before removing duplicates, still in exponential time.}, and nondeterministically guess $k$ of them, then verify that they form a vertex cover.
\end{proof}

(There is nothing special about $k$-\probref{VertexCover} here; it is easy to come up with similar \NEXP algorithms for other `$k$-bounded' versions of our problems.)

This is worse than the \PSPACE bound we gave for \probref{VertexCover}. In fact, we will shortly show both problems are hard for these classes.

\subsection{Lower bounds}

It is easy to see that all our problems are \PSPACE-hard. This is because all the operations from Lemma~\ref{lem:pspace-basics} are \PSPACE-hard: for emptiness checking this follows from \PSPACE-hardness of satisfiability for the first-order theory of equality, and the other parts easily reduce to that. As a result, Theorems~\ref{thm:many-pspace}-\ref{thm:pspace-last} settle the complexity of all our decidable problems except \probref{3-Sat} and the $k$-bounded versions mentioned above.

To understand these, let us begin by recapping \emph{succinct problems}, whose study was initiated in \cite{Galp84}.

The idea (following \cite{LB90, BLT92}) is as follows: given a decision problem $A$, whose inputs we see (via a suitable encoding) as binary strings, the `\emph{succinct}' problem $sA$ takes as input \emph{succinct representations} of those binary strings. An instance of $sA$ is positive precisely if the binary string it represents is a positive instance of $A$.

A succinct representation of a binary string $x$ is a Boolean circuit, which on input the binary representation of $i \in \mathbb{Z}^+$ should output (i) whether $i \leq \lvert x \rvert$ (the length of $x$) and (ii) whether the $i$th bit of $x$ is a $1$. If a string is highly regular, its succinct representation can be substantially smaller.

See \cite{PY86,LB90,BLT92,Helmut98} for more on the general theory of succinct problems. For our purposes it is enough to know that the succinct versions of most \NP-complete problems are \NEXP-complete. In particular:

\begin{thm} The succinct versions of the classical problems \probname{3-Sat} and \probname{VertexCover} are \NEXP-hard.
\end{thm}
\begin{proof} The techniques of \cite{PY86, LB90, BLT92} apply to most \NP-hard problems. To verify that they apply to \probname{3-Sat} and \probname{VertexCover}, it suffices to check standard reductions from appropriate formulations of e.g. \probname{Satisfiability}, or \probname{\mbox{$3$}-Colorability}, and verify that these lie within some sufficiently weak class\footnote{\ The most general being poly-logarithmic time reductions (see~\cite{Helmut98, EGM97}).}. This is easy to do. Alternatively, we could directly reduce between the succinct versions; this is also not difficult.
\end{proof}

Finally, we reduce these succinct problems to our corresponding definable ones. It is interesting to note that in fact the reduction goes only through finite instances. 
Here again, there is nothing special about \probname{VertexCover}; very similar reductions apply to most of the $k$-bounded variants of our problems.

\begin{thm}\label{thm:nexphard} \probref{3-Sat} and $k$-\probref{VertexCover} are \NEXP-hard.
\end{thm}
\begin{proof} By reduction from the corresponding succinct problems. We argue for $k$-\probref{VertexCover}; \probref{3-Sat} is similar.

Take as input an instance of the succinct version of \probname{VertexCover}. We may assume\footnote{\ We defined the succinct problem slightly differently, but may produce the items above in polynomial time.} this consists of:
\begin{itemize}
\item a set of vertices, represented as a Boolean circuit $\mathcal{C}_V$ with $n$ input gates,
\item a set of edges, represented as Boolean circuit $\mathcal{C}_E$ with $2n$ input gates,
\item a number $k$, represented in binary.
\end{itemize}
The first two components encode a graph $G = (V,E)$ with vertices $V\subseteq \{0,1\}^n$. The content of our reduction will be to construct, in polynomial time, a definable graph isomorphic to $G$. 

First, we transform circuits into formulas: in polynomial time produce from $\mathcal{C}_V$ an existentially quantified Boolean formula 
\[
\varphi_V = \exists q_1 \cdots \exists q_m. \chi_V(p_1, \ldots, p_n, q_1, \ldots, q_m)
\]
(with $\chi_V$ quantifier-free), with $n$ free propositional variables $p_1, \ldots, p_n$ (and variables $q_i$ corresponding to the gates of $\mathcal{C}_V$), such that $\varphi_V$ holds exactly when $\mathcal{C}_V$ returns $1$. Similarly we can produce a formula $\varphi_E$, with $2n$ free variables, from $\mathcal{C}_E$.

Next, we define the vertex set of our graph. The general idea is that a binary sequence $v\in\{0,1\}^n$ can be encoded as a pair of tuples of atoms $(\bar{a},\bar{b})$, both of length $n$, where all the atoms are pairwise distinct except that $a_i=b_i$ if $v_i=1$. In fact, the {\em orbit} of $(\bar{a},\bar{b})$ fully describes the sequence $v$, and so a set of sequences, such as the one defined by the formula $\varphi_V$, will be encoded as a set of orbits.

To this end, first define a formula $\zeta(\bar{x},\bar{y})$, over the language of equality, with $2n$ atom variables, which says that all the values are pairwise distinct except perhaps $x_i$ and $y_i$ for any $i$. This is simply a conjunction of quadratically many inequalities. Then replace the propositional variables in $\phi_V$ with equalities between atom variables like so:
\[
\psi_V(\bar{x},\bar{y}) = \exists z_1\exists z'_1 \cdots \exists z_m\exists z'_m. \chi_V(x_1=y_1,\, \ldots, x_n=y_n,\, z_1=z'_1,\, \ldots, z_m=z'_m).
\]
This is a formula over the language of equality, with $2n$ free variables. Then the conjunction $\zeta(\bar{x},\bar{y})\land\psi_V(\bar{x},\bar{y})$ holds for those $(\bar{a},\bar{b})\in(\atoms^n)^2$ which encode binary sequences for which $\phi_V$ holds.

We now want to quotient this subset of $(\atoms^n)^2$ by the relation of being in the same orbit. For this, it is easy to write a formula $\xi(\bar{x},\bar{y},\bar{x}',\bar{y}')$, with $4n$ variables, such that $\xi(\bar{a},\bar{b},\bar{a}',\bar{b}')$ holds exactly when $(\bar{a},\bar{b})$ and $(\bar{a}',\bar{b}')$ are in the same equivariant orbit of $(\atoms^{n})^2$. This formula is a conjunction of quadratically many equivalences between atom equalities.

We can now express the set of vertices of our definable graph as:
\[
	\big\{ \{ (\bar{x}',\bar{y}') \mid \bar{x}',\bar{y}'\in\atoms^n,\ \xi(\bar{x},\bar{y},\bar{x}',\bar{y}') \} \mid \bar{x},\bar{y}\in\atoms^n,\ \zeta(\bar{x},\bar{y})\land\psi_V(\bar{x},\bar{y}) \big\}.
\]
The expression for the set of edges is similar, with the formula $\phi_V$ replaced by $\phi_E$ and the arities of atom tuples changed accordingly. This produces a definable graph isomorphic to $G$, in time polynomial with respect to the size of the circuits $\mathcal{C}_V$ and $\mathcal{C}_E$.
\end{proof}

This settles the complexity of the decidable versions of Karp's problems on definable instances, under our chosen representation of inputs. It should be said that this picture is rather rough: the \PSPACE-hardness of determining emptiness and equality overshadows any potential finer complexity considerations. One way to avoid this problem is to fix, as a parameter, the maximal size of the least support of any component of the input. This approach has been proposed in~\cite{BT18} to describe the class of `fixed-dimension polynomial time'. We leave it for future work to determine which of Karp's problems, under various input representations, become fixed-dimension \NP.


\section{Conclusion}

We chose to consider Karp's 21 problems simply because they are famous and well-studied. But, in a sense, the choice was principled: by fixing the set in advance, we forced our hand to look at a whole range of possible properties, in the definable setting, of problems which all behave similarly over finite structures. So it is interesting how different these properties turned out to be: some problems (e.g.~\probref{ExactCover}) are undecidable for nontrivial reasons, some (e.g.~\probref{3-Sat}) are decidable by easy algorithms which are correct for deep reasons, some (e.g.~\probref{SetCovering}) are decidable by algorithms which are correct for mundane reasons, and some (e.g.~\probname{Knapsack}) do not seem to generalise to the definable setting at all. 

We may project this picture onto the tree of reductions that Karp originally used for proving~\NP-hardness of the finite versions of these problems. We reproduce that tree in Figure~\ref{fig:karps-tree}, where we mark the problems decidable on definable structures in green, and the undecidable ones in red. The region of weighted problems, whose definable reformulations either are missing or look rather degenerate (see Section~\ref{sec:weighted}), is grayed out. 

\begin{figure}
\centering
\[\resizebox{\textwidth}{!}{\xymatrix{
& & \probref[Red]{CNFSat}\ar[ld]\ar[d]\ar[rd] \\
& \probref[Green]{Clique}\ar[d] & \probref[Red]{0-1-ILP} & \probref[Green]{3-Sat}\ar[d] \\
& \probref[Green]{VertexCover}\ar[ldd]\ar[ld]\ar[d]\ar[rd] & & \probref[Green]{Colorability} \ar[d]\ar[rd] \\
\hyperref[prob:FeedbackVertexSet]{\txt{\probname[Green]{Feedback}\\ \probname[Green]{VertexSet}}} &
\hyperref[prob:DirectedHamiltonicity]{\txt{\probname[Red]{Directed}\\ \probname[Red]{Hamiltonicity}}}\ar[d] &
\probref[Green]{SetCovering} & \probref[Red]{ExactCover}\ar[ldd]\ar[ld]\ar[ld]\ar[d]\ar[rd] &
\hyperref[prob:CliqueCover]{\txt{\probname[Green]{Clique}\\ \probname[Green]{Cover}}} \\
\hyperref[prob:FeedbackArcSet]{\txt{\probname[Green]{Feedback}\\ \probname[Green]{ArcSet}}} &
\hyperref[prob:UndirectedHamiltonicity]{\txt{\probname[Red]{Undirected}\\ \probname[Red]{Hamiltonicity}}} &
\probref[Red]{HittingSet} 
& \probname[lightgray]{Knapsack}\ar[d]\ar[rd] & 
\hyperref[prob:SteinerTree]{\txt{\probname[Green!50]{Steiner}\\ \probname[Green!50]{Tree}}}\\
& & \probref[Red]{3DMatching} & \probname[lightgray]{Sequencing} & \probname[lightgray]{Partition}\ar[d]\\
& & & & \probref[Green!50]{MaxCut}
\save "5,4"+<-35pt,12pt>."7,5"*+<3pt,3pt>[F.:<3pt>]\frm{} \restore
\save "7,5"+<-95pt,-14pt>*\txt{\color{Gray}{\tiny weighted problems}} \restore
}}\]
\caption{Karp's tree of reductions~\cite[Fig.~1]{Karp72}, annotated}
\label{fig:karps-tree}
\end{figure}
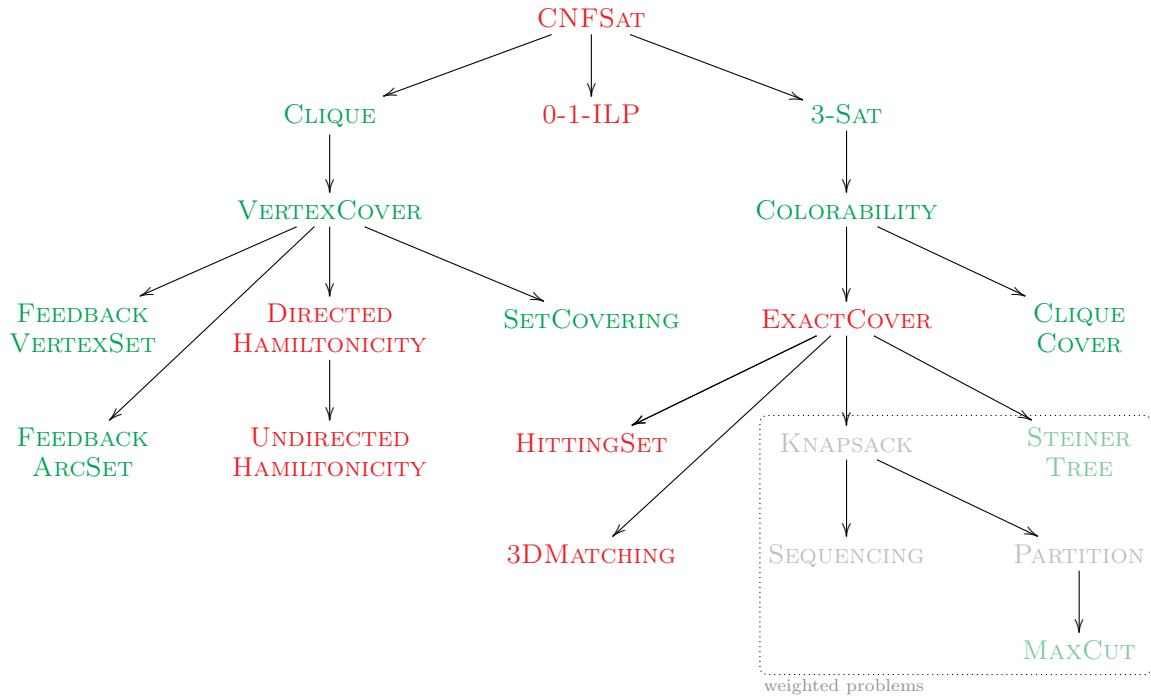

It is interesting to see how the problems that become undecidable on definable structures are scattered over the tree. Certainly Karp could have chosen different a different structure of reductions, but it is natural to think that the ones used in~\cite{Karp72} and shown in Figure~\ref{fig:karps-tree} are some of the first, and perhaps the simplest, that come to mind. It is remarkable, then, how few of them remain in force in the definable setting: we were able to reuse a couple of them with no change (e.g.~\probref{ExactCover} to~\probref{HittingSet} in Theorem~\ref{thm:exactcover-hittingset}), but several others (e.g.~\probref{CNFSat} to~\probref{3-Sat}) break down in fundamental ways. Even for those problems that turned out to be undecidable, we felt compelled to create a very different tree of reductions, shown in Figure~\ref{fig:our-tree}.

\begin{figure}
\centering
\[\resizebox{\textwidth}{!}{\xymatrix{
& & \probref[Red]{CNFSat} \\
& \phantom{\probref[Green]{Clique}} & \probref[Red]{0-1-ILP} & \phantom{\probref[Green]{3-Sat}} \\
& \phantom{\probref[Green]{VertexCover}} & & \phantom{\probref[Green]{Colorability}} \\
\phantom{\hyperref[prob:FeedbackVertexSet]{\txt{\probname[Green]{Feedback}\\ \probname[Green]{VertexSet}}}} &
\hyperref[prob:DirectedHamiltonicity]{\txt{\probname[Red]{Directed}\\ \probname[Red]{Hamiltonicity}}}\ar[d]_{\text{Thm.~\ref{thm:undecidability-undirected-hamiltonicity}}} &
\phantom{\probref[Green]{SetCovering}} & \probref[Red]{ExactCover}\ar[ldd]^{\text{Cor~\ref{coro:exact-tuple-cover}, Thm.~\ref{thm:3dmatching}}}\ar[ld]_{\text{Thm~\ref{thm:exactcover-hittingset}}}\ar@/_2pc/[ll]_(.2){\text{\qquad Cor~\ref{coro:exact-tuple-cover}, Thm.~\ref{thm:undecidability-directed-hamiltonicity}}} &
\phantom{\hyperref[prob:CliqueCover]{\txt{\probname[Green]{Clique}\\ \probname[Green]{Cover}}}} \\
\phantom{\hyperref[prob:FeedbackArcSet]{\txt{\probname[Green]{Feedback}\\ \probname[Green]{ArcSet}}}} &
\hyperref[prob:UndirectedHamiltonicity]{\txt{\probname[Red]{Undirected}\\ \probname[Red]{Hamiltonicity}}} &
\probref[Red]{HittingSet}\ar@/^4.5pc/[uuuu]^(.7){\text{Thm.~\ref{thm:undecidability-sat}}}\ar[uuu]_(.8){\text{Thm.~\ref{thm:undecidability-01ILP}}}
& \phantom{\probname[lightgray]{Knapsack}}& 
\phantom{\hyperref[prob:SteinerTree]{\txt{\probname[Green!50]{Steiner}\\ \probname[Green!50]{Tree}}}}\\
& & \probref[Red]{3DMatching} & \phantom{\probname[lightgray]{Sequencing}} & \phantom{\probname[lightgray]{Partition}}
}}\]
\caption{Our tree of reductions for the undecidable problems}
\label{fig:our-tree}
\end{figure}
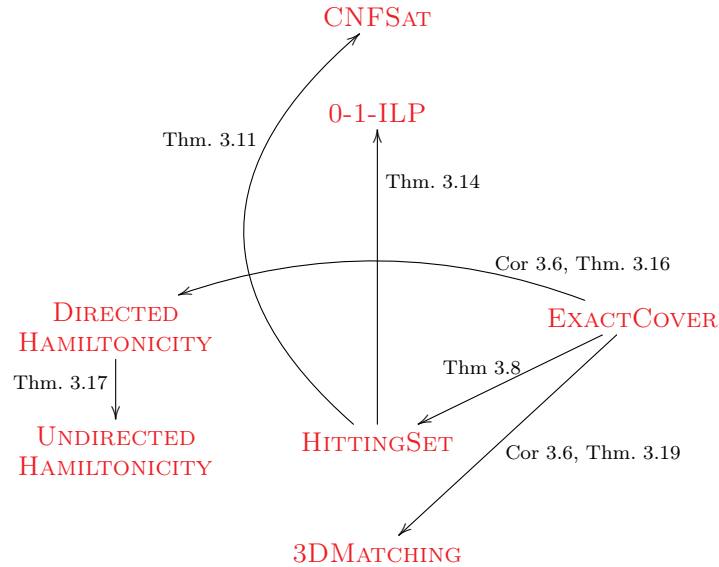

Together with the analysis of decidable problems in Sections~\ref{sec:amenability}-\ref{sec:other}, and their complexity settled in Section~\ref{sec:complexity}, this paints quite an interesting landscape, and shows that we now have a small arsenal of techniques to analyse simple computational problems in the definable setting.

\subsection*{Acknowledgments} We are grateful to M.~Boja\'{n}czyk, A.~Ghosh, P.~Hofman, S.~Lasota and A.~Mylet for useful discussions and comments.

\bibliographystyle{alphaurl}
\bibliography{bib}

\end{document}